\documentclass[11pt]{article}
\usepackage{amsmath,amssymb,amsxtra}
\usepackage{mathrsfs}
\usepackage{float,verbatim}
\usepackage{threeparttable,booktabs}
\usepackage{graphicx}
\usepackage{epstopdf}
\usepackage{mathtools}
\usepackage{subfig}
\usepackage[ruled,vlined]{algorithm2e}
\usepackage{color}
\usepackage{url}
\usepackage[toc,page]{appendix}
\graphicspath{{./fig/}}

\usepackage[nohead,margin=1.1in]{geometry}

\usepackage{array}

\newtheorem{theorem}{Theorem}[section]

\newtheorem{remark}{Remark}[section]
\newtheorem{prop}{Proposition}[section]
\newtheorem{lemma}{Lemma}[section]
\newtheorem{corollary}{Corollary}[section]
\newtheorem{assumption}{Assumption}[section]

\numberwithin{equation}{section}
\newtheorem{proof}{Proof}[section]

\allowdisplaybreaks
\newcommand{\E}{\mathbb{E}}
\newcommand{\Prob}{\mathbb{P}}

\title{Early stopping of stochastic variance reduced gradient for linear inverse problems by the discrepancy principle\thanks{B. Jin is supported by Hong Kong RGC General Research Fund (14306824) and ANR / Hong Kong RGC Joint Research Scheme (A-CUHK402/24), NSFC / RGC Joint Research Scheme (N\_CUHK446/25) and a start-up fund from The Chinese University of Hong Kong.}}

\author{Bangti Jin\thanks{Department of Mathematics, The Chinese University of Hong Kong, Shatin, N.T., Hong Kong (email: \texttt{b.jin@cuhk.edu.hk, zehuizhou@cuhk.edu.hk})}\and Zehui Zhou\footnotemark[2]}

\date{}

\begin{document}

\maketitle

\begin{abstract}
Stochastic variance reduced gradient (SVRG) is a variant of stochastic gradient descent and is a promising iterative method for solving large-scale inverse problems. 
Nevertheless, the development of theoretically grounded \textit{a posteriori} stopping rules for SVRG remains an open challenge. 
In this work, we provide a convergence analysis of SVRG equipped with the discrepancy principle, the most well-known \textit{a posteriori} stopping rule, for solving a class of linear inverse problems in Hilbert spaces. We establish the regularizing property of SVRG, and moreover, under suitable source conditions, we derive convergence rates of SVRG iterates. To the best of our knowledge, these are the first convergence rate results of any stochastic iterative method for inverse problems under the \textit{a posteriori} stopping rule. The theoretical findings are supported by numerical experiments. 
\end{abstract}


\date{}
\maketitle

\section{Introduction}

This work is concerned with using stochastic iterative methods for solving linear inverse problems in Hilbert spaces:
\begin{equation}\label{eqn:lininv}
A x=y^\dag,
\end{equation}
where the operator $A: X \rightarrow Y =Y_1\times \cdots \times Y_n$ represents the data formation mechanism and is given by $A x := (A_1 x,\cdots,A_n x)^\top$ for all $x\in X$, with bounded linear operators $A_i: X \rightarrow Y_i$ between Hilbert spaces $X$ and $Y_i$ equipped with norms $\|\cdot\|_X$ and $\|\cdot\|_{Y_i}$, respectively, and the superscript $\top $ denoting the vector transpose. 
$y^\dag=(y^\dag_1,\cdots,y^\dag_n)^\top = A x^\dag \in Y$ denotes the exact data with $x^\dag$ being the minimum-norm solution relative to the initial guess $x_0$, cf. \eqref{eqn:min-norm} below for the definition, 
and $x\in X$ denotes the unknown signal of interest.
In practice, we only have access to noisy data $$y^\delta =(y^\delta_1,\cdots,y^\delta_n)^\top =y^\dag +\xi,$$
where $\xi = (\xi_1,\cdots,\xi_n)^\top\in Y$ denotes the noise in the data with a noise level $\delta =\|\xi\|_Y:=\sqrt{\sum_{i=1}^n\|\xi_i\|_{Y_i}^2}$.
Linear inverse problems of the form \eqref{eqn:lininv} arise naturally in diverse practical applications, e.g., image deblurring \cite{BiemondLagendijk:1990} and computed tomography \cite{HermanLentLutz:1978}.

Stochastic gradient descent (SGD) \cite{RobbinsMonro:1951,JinLu:2019} and its variants, e.g., stochastic variance reduced gradient (SVRG) \cite{JohnsonZhang:2013,ZhangMahdaviJin:2013,JinZhouZou:2022ip}, are highly popular stochastic iterative methods for solving large-scale inverse problems, due to their excellent scalability with respect to data size. 
However, their mathematical analysis for inverse problems from the perspective of regularization theory has not been fully explored.
We refer interested readers to the recent surveys  \cite{EhrhardtKereta:2025,JinXiaZhou:2025} for existing results on stochastic iterative methods for inverse problems. To illustrate the idea, consider the following optimization problem 
\begin{equation*}
\min_{x\in X} J(x)
\end{equation*}
with the objective function $J(x)$ given by
\begin{equation*}
 J (x) = \frac{1}{2n}\|A x-y^\delta\|_Y^2 = \frac{1}{n}\sum_{i=1}^n f_i(x), \quad \mbox{with } f_i(x)=\frac{1}{2} \|A_ix-y_i^\delta\|_{Y_i}^2.
\end{equation*}
Given an initial guess $x_0^\delta\equiv x_0\in X$, SGD reads
\begin{equation*}
x_{k+1}^\delta = x_k^\delta - \eta_k f'_{i_k}({ x_k^\delta}), \quad k=0,1,\cdots,
\end{equation*}
and SVRG reads
\begin{equation}\label{eqn:SVRG}
x_{k+1}^\delta =x_k^\delta -\eta_k \big( f'_{i_k}(x_k^\delta)- f'_{i_k}(x_{[k/M]M}^\delta)+ J'(x_{[k/M]M}^\delta)\big), \quad k=0,1,\cdots,
\end{equation}
where the index $i_k$ is sampled uniformly at random from the set $\{1,\ldots,n\}$, $\{\eta_k\}_{k\geq 0}\subset (0,\infty)$ denotes the step size schedule, and $x_{[k/M]M}^\delta$ is an anchor point with $M$ being the frequency of computing the full gradient and $[\cdot]$ taking the integral part of a real number. The efficiency of the variance reduction step in SVRG depends on the frequency $M$, which in practice was suggested to be $2n$ and $5n$ for convex and nonconvex optimization, respectively  \cite{JohnsonZhang:2013}.
In practice, there are other choices of the anchor point, leading to different variants of SVRG.
In this work, we focus on the version given in Algorithm \ref{alg:svrg}, where $A^*$ and $A_i^*$ denote the adjoints of  $A$ and $A_i$, respectively.

\medskip
\begin{algorithm}[H]
\SetAlgoLined
Set initial guess $x_0^\delta=x_0$, frequency $M$, and step size schedule $\{\eta_k\}_{k\geq 0}$\\
 \For{$K=0,1,\cdots$}{
Compute $ {g_K = J'(x_{KM}^\delta)=\tfrac{1}{n}A^*(A x_{KM}^\delta-y^\delta) }$\\
\For{$t=0,1,\cdots,M-1$}{
Draw {$i_{KM+t}$} i.i.d. uniformly from $\{1,\cdots,n\}$\\
Update $x_{KM+t+1}^\delta =x_{KM+t}^\delta -\eta_{KM+t} \big(A_{i_{KM+t}}^*A_{i_{KM+t}}( x_{KM+t}^\delta-x_{KM}^\delta)+ g_K\big)$
}
Check the stopping criterion}
\caption{Stochastic Variance Reduced Gradient (SVRG) for problem \eqref{eqn:lininv}.\label{alg:svrg}}
\end{algorithm}

\medskip

To avoid uncontrolled growth of noise propagation in the iteration \eqref{eqn:SVRG}, early stopping rules are needed for obtaining meaningful approximations.  \textit{A priori} stopping rules that depend on the degree of regularity of the reference solution have been studied for SGD and SVRG for linear inverse problems through the lens of regularization theory \cite{JinLu:2019,JinZhouZou:2020siopt,JinZhouZou:2021siuq,JinZhouZou:2022ip,LuMathe2022}. 
However, in practice, the degree of regularity of the exact solution $x^\dag$ is unknown, making the \textit{a priori} rule practically infeasible. 
Thus, it is imperative to develop \textit{a posteriori} stopping rules that depend only on the noisy data $y^\delta$ and the noise level $\delta$.
The discrepancy principle due to the Russian mathematician Vladimir Morozov \cite{Morozov:1966} is the most popular \textit{a posteriori} stopping rule. It chooses the stopping index $k(\delta)$ by
\begin{align}\label{eqn:discrepancy_1}
k(\delta)=\min\{k\in \mathbb{N}:\;\|Ax_{k}^\delta-y^\delta\|\leq \tau \delta\},
\end{align}
with a fixed constant $\tau>1$. 
Given the stopping index $k(\delta)$, one central issue through the lens of regularization theory is to study the convergence and convergence rates of the approximation $x_{k(\delta)}^\delta$ as the noise level $\delta$ tends to zero. This important issue remains open for SVRG equipped with the discrepancy principle. Note that only recently has the discrepancy principle \eqref{eqn:discrepancy_1} been investigated and shown to be valid for SGD when solving linear inverse problems in Hilbert spaces \cite{JahnJin:2020}.
We also refer the readers to \cite{HuangJinLuZHang:2025} for an \textit{a posteriori} stopping rule that adapts the discrepancy principle to stochastic mirror descent (SMD) for solving inverse problems in Banach spaces. 
However, the convergence rates of these algorithms remain missing. 

In this work, we investigate the convergence of SVRG stopped by the discrepancy principle \eqref{eqn:discrepancy_1} at the anchor points $KM$, with the stopping index $k(\delta)$ determined by
\begin{align}\label{eqn:discrepancy_2}
k(\delta)=M\min\{K\in \mathbb{N}:\;\|Ax_{KM}^\delta-y^\delta\|\leq \tau \delta\}.
\end{align}
Without loss of generality, we may assume that $\|Ax_0^\delta-y^\delta\|\geq \tau \delta$ so that $k(\delta)\geq M$.
One advantage of SVRG over SGD and SMD is that it does not incur any extra cost when using the rule \eqref{eqn:discrepancy_2}. 
Indeed, unlike SGD or SMD \cite{JahnJin:2020,HuangJinLuZHang:2025}, SVRG periodically computes the residual $A x_{KM}^\delta-y^\delta$ at the anchor point $x_{KM}^\delta$ for every $M$ iterations, making it well suited to use the rule \eqref{eqn:discrepancy_2}. 
Moreover, its built-in variance reduction mechanism \cite{JinZhouZou:2022ip} leads to smoother error and residual trajectories, which potentially enhances the numerical stability of the rule \eqref{eqn:discrepancy_2}.
In this work, we prove that SVRG iterates with a suitable step size schedule converge to the minimum-norm solution $x^\dag$ at a certain rate in terms of the noise level $\delta$ for any frequency $M$, when stopped by the rule \eqref{eqn:discrepancy_2}.
The convergence rate is (nearly) optimal for nonsmooth solutions and is comparable with existing rates for SGD \cite{JinLu:2019,JinZhouZou:2020siopt} and SVRG \cite{JinZhouZou:2022ip,JinChen:2024} with \textit{a priori} stopping rules.

The discrepancy principle \eqref{eqn:discrepancy_1} has been investigated in the works \cite{JahnJin:2020} and \cite{JinChen:2024} for SGD and SVRG, respectively, for linear inverse problems in Hilbert spaces. 
Specifically, the work \cite{JahnJin:2020} proves both the finite-iteration termination property (heavily depending on the decaying step sizes) and regularizing property in high probability for SGD, while \cite{JinChen:2024} presents the almost sure finite-iteration termination for SVRG. Compared with the existing works \cite{JahnJin:2020,JinChen:2024}, we contribute to the convergence analysis both in high probability and in the uniform sense of SVRG stopped by the discrepancy principle; see Theorems \ref{thm:dp_E} and \ref{thm:dp} for the finite-iteration termination property and convergence rates in terms of the noise level $\delta$, and Theorem \ref{thm:regularizing} for the regularizing property. 
The obtained convergence rates are (nearly) order optimal for nonsmooth solutions. To the best of our knowledge, these represent the first convergence rate results for any stochastic iterative method equipped with the discrepancy principle.

The rest of the work is organized as follows. In Section \ref{sec:main}, we present and discuss the main results, i.e., Theorems \ref{thm:dp_E}, \ref{thm:dp} and \ref{thm:regularizing}, including the finite-iteration termination properties, convergence rates, and the regularizing property, for SVRG with the discrepancy principle \eqref{eqn:discrepancy_2}. We present the proof of the main results in Section \ref{sec:conv}. Then in Section \ref{sec:num}, we present several numerical experiments to complement the analysis, which indicate the optimality of SVRG and its advantages over the Landweber method. 
Finally, we conclude this work with further discussions in Section \ref{sec:conc}. In the three appendices, we collect several technical estimates and intermediate results. Below, we surppress the subscripts in the norms and inner products, as the spaces are clear from the context. The notation $(\cdot,\cdot)$ denotes the inner product.

\section{Main results and discussions}\label{sec:main}
In this section, we present the main results of the work. 
First, we state the assumptions on the step size $\eta_j$ and the reference solution $x^\dag$, which is taken to be the unique minimum-norm solution relative to the initial guess $x_0$, given by
\begin{equation}\label{eqn:min-norm}
  x^\dag = \arg\min_{x\in X: A x = y^\dag} \|x-x_0\|,
\end{equation}
for analyzing the convergence of SVRG for linear inverse problems. These conditions are often employed to establish the convergence rates of regularized solutions \cite{HankeNeubauerScherzer:1995,JinZhouZou:2020siopt,JinZhouZou:2021siuq,JinZhouZou:2022ip}. The notation $\|A_i\|$ denotes the operator norm of $A_i$, and $\mathcal{N}(A)$ denotes the null space of $A$. 
\begin{assumption}\label{ass}
The following assumptions hold.
\begin{itemize}
\item[$\rm(i)$] The step size $\eta_j = c_0$, $j=0,1,\cdots$, with $c_0\leq L^{-1}$, where $L:= \max_{1\leq i\leq n}\|A_i\|^2$.
\item[$\rm(ii)$] There exist some $\nu>0$ and $w\in \mathcal{N}(A)^\perp$ such that $x^\dag-x_0=B^\nu w$ and $\|w\|<\infty$, with $B=n^{-1}A^*A$ and $\mathcal{N}(A)^\perp$ being the orthogonal complement of $\mathcal{N}(A)$.
\end{itemize}
\end{assumption}

The constant step size schedule in (i) is commonly employed by SVRG \cite{JohnsonZhang:2013}. 
(ii) is known as the source condition  \cite{EnglHankeNeubauer:1996}, which implicitly assumes a certain degree of regularity on the initial error $x^\dag-x_0$ and is crucial for deriving convergence rates for iterative methods. 
Without a source condition, the convergence of regularized solutions  may be arbitrarily slow \cite{EnglHankeNeubauer:1996}. 

Let $\mathcal{F}_k$ be the filtration generated by the random indices $\{i_0,i_1,\ldots,i_{k-1}\}$, $\mathcal{F}={\bigvee_{k=1}^\infty} \mathcal{F}_k$, $(\Omega,\mathcal{F},\Prob)$ be the associated probability space, and $\E[\cdot]$ denote taking the expectation with respect to the filtration $\mathcal{F}$. 
The SVRG iterate $x_k^\delta$ is random but measurable with respect to $\mathcal{F}_k$. 
Now, we state the main results on the finite-iteration termination property and convergence rates of SVRG with the discrepancy principle \eqref{eqn:discrepancy_2}. The convergence rate analysis relies crucially on the observation that the variance component of the residual contributes only marginally for large iteration index $k$. The proofs of these results are given in Section \ref{sec:conv}.  Let
\begin{align*}
\overline{C_0}&:=\max\big(3^{-1}\|A\|^{-1}(5Ln^{-1}M)^{-\frac12},9^{-1}L^{-1}(15+7\ln M)^{-1}\big),\\
C_0&:=\big(14^2 \sqrt{L}M\|A\| \ln(2e^2n\sqrt{L}\|A\|^{-1})\big)^{-1}.
\end{align*}
Below the notation $\lfloor\cdot\rfloor$ and $\lceil\cdot\rceil$ denote the floor and ceiling functions, respectively. Without loss of generality, in the analysis, we assume $\delta\leq 1$ (which can be easily achieved by properly rescaling the inverse problems) and $M\geq 2$ (When $M=1$, SVRG reduces to the Landweber method). 
\begin{theorem}\label{thm:dp_E}
Let Assumption \ref{ass} hold with $c_0<\overline{C_0}$, and let the stopping index $k(\delta)$ be chosen by the discrepancy principle \eqref{eqn:discrepancy_2} with $\tau>1$.
Then with the convention $(\delta^{-1})^0:=\ln (\delta^{-1})$ and the constant $c_{\nu}=(\nu+\frac12)^{\nu+\frac12}c_0^{-(\nu+\frac12)}$, there holds
\begin{align*}
\lim_{\delta\to 0^+}\Prob\Bigg(k(\delta)\leq M\Bigg\lceil M^{-1}\bigg\lceil\bigg(\frac{2\sqrt{n}c_\nu\|w\|}{\tau-1}\bigg)^{\frac{2}{1+2\nu}}\delta^{-1}\big(\delta^{-1}\big)^{\max(\frac{2}{1+2\nu}-1,0)}\bigg\rceil \Bigg\rceil\Bigg)=1.
\end{align*}
Further, if $\nu\in(0,\frac12]$, for any $s\geq \max\big(1,\frac13(1+4\frac{\ln M}{\ln n})\big)$, there exists an event $Q$ satisfying $\Prob(Q)\geq 1-\frac12 \delta^{2(3s-2)}$ such that 
\begin{align*}
\E[\|x_{k(\delta)}^\delta-x^\dag\|^2|Q]^\frac12 \leq \overline{C}_{w,\nu,\tau}^* \max\big(s,s\big(\tfrac{\tau-1}{2c_\nu \|w\|}\big)^{3s},n^{-\frac{1}{2}}M^{2\nu}\big)^\frac{1}{1+2\nu} M^\frac{1}{1+2\nu} \delta^\frac{2\nu}{1+2\nu}\ln^\frac{1}{1+2\nu}(\delta^{-1}+M),
\end{align*}
where  $\overline{C}_{w,\nu,\tau}^*$ depends on $w$, $\nu$, $\tau$ and $(\tau-1)^{-1}$, but is independent of $s$, $n$, $M$ and $\delta$. 
\end{theorem}

\begin{theorem}\label{thm:dp}
Let Assumption \ref{ass} hold with $c_0<C_0$, and let the stopping index $k(\delta)$ be chosen by the discrepancy principle \eqref{eqn:discrepancy_2} with $\tau>1$.
Then there exists some  $c^{*}$ independent of $w$, $n$, $k(\delta)$ and $\delta$ such that, for any $\epsilon\in(0,\frac12)$, 
\begin{align*}
k(\delta)\leq \left\{\begin{array}{cc}
M\Big\lceil M^{-1}\Big(\big(\frac{n(c_\nu\|w\|+c^{*})}{(\tau-1) \delta}\big)^{\frac{2}{1+2\nu}}+1\Big)\Big\rceil,     & \nu<\frac12, \\
M\Big\lceil M^{-1}\Big(\big(\frac{n(c_\nu\|w\|+c^{*})}{e\epsilon(\tau-1) \delta}\big)^{\frac{1}{1-\epsilon}}+1\Big)\Big\rceil,     & \nu\geq \frac12.
\end{array}\right.
\end{align*}
Further, if $\nu\in(0,\frac12)$, then there holds
\begin{align*}
\|x_{k(\delta)}^\delta-x^\dag\|
\leq C_{w,\nu,\tau}^* 
n^\frac{1-2\nu}{2+4\nu} M^\frac{1}{1+2\nu} 
\delta^{\frac{2\nu}{1+2\nu}},
\end{align*}
where $C_{w,\nu,\tau}^*$ depends on $w$, $\nu$, $\tau$ and $(\tau-1)^{-1}$, but is independent of $k(\delta)$, $n$, $M$ and $\delta$.
\end{theorem}

\begin{remark}
The convergence rates in Theorems \ref{thm:dp_E} and \ref{thm:dp} are largely comparable with that for SVRG equipped with the \textit{a priori} parameter choice in \cite{JinZhou:2026}, up to the log factor, and also that for the Landweber method \cite{EnglHankeNeubauer:1996}. These results indicate the {\rm(}near{\rm)} optimality of the error bounds for the SVRG with the discrepancy principle \eqref{eqn:discrepancy_2}. 
The assumption  $\nu\in(0,\frac12]$ ensures the boundedness of $\|B^{\frac12-\nu}\|$, which is necessary for deriving the convergence rate via the interpolation inequality. In Theorem \ref{thm:dp}, an additional boundedness condition involving $(1-2\nu)^{-1}$ is imposed, which further requires $\nu \in (0,\frac12)$.
The restriction $\nu\in(0,\frac12)$ is also present for the standard SGD and SVRG, and it is attributed to the presence of the variance component of the SVRG, which leads to the saturation phenomenon. 
\end{remark}

\begin{remark}
Note that the convergence rates in the mean squared sense and uniform sense require different conditions on the step size: the latter imposes a more stringent condition. 
The condition $c_0\leq {C_0}$ is used for the estimate $\|A(x^\delta_{KM+t}-x^\delta_{KM})\|\leq (c_1+c_2\delta)M(KM+t+M)^{-1}$ in Lemma A.3, while $\|g_K\|=n^{-1}\|A^* r_{KM}^\delta\|$ with $\|r_{KM}^\delta\|\leq n (c_\nu\|w\|+Mc^{*})(KM)^{-(\frac12+\nu)}+\delta$ when $\nu\in(0,\frac12)$. 
Thus, the influence of the stochastic gradient gap $A^*_{i_{KM+t}}A_{i_{KM+t}}(x^\delta_{KM+t}-x^\delta_{KM})$ becomes weaker relative to the full gradient $g_K$ as $k$ grows. This is crucial for deriving a uniform and deterministic upper bound for $\|e_{k(\delta)}^\delta\|$. 
\end{remark}
In the absence of the source condition in Assumption \ref{ass}(ii), both the finite-iteration termination property and the regularizing property of SVRG remain valid. That is, the discrepancy principle \eqref{eqn:decom_dp} is a valid \textit{a posteriori} stopping rule for SVRG. 
\begin{theorem}\label{thm:regularizing} Let Assumption \ref{ass}{\rm(i)} hold, and let the stopping index $k(\delta)$ be chosen by the discrepancy principle \eqref{eqn:discrepancy_2} with $\tau>1$. Then for any $\epsilon_0\in(0,1)$, there exist some $w_{\epsilon_0}$ that depends only on $\epsilon_0$ and some $c^{*}$ independent of $\epsilon_0$, 
$n$, $k(\delta)$ and $\delta$ such that the following statements hold.
 \begin{itemize}
    \item[{\rm(i)}] If $c_0<C_0$, then for any $\epsilon\in(0,\frac12)$, there hold
\begin{align*}
&k(\delta)\leq \Bigg(\frac{n  \epsilon_0^2+2\epsilon^{-1}c_0n(\tau-1)(c_0^{-1}\|w_{\epsilon_0}\|+M}c^{*}) \delta{2c_0(\tau-1)^2 \delta^2}\Bigg)^\frac{1}{1-\epsilon}+M+1,\\
&\lim_{\delta\to 0^+}\|x_{k(\delta)}^\delta-x^\dag\|\leq \tau(\tau-1)^{-1}\epsilon_0.
\end{align*}
\item[{\rm(i)}] If $c_0<\overline{C_0}$, then for any $p_0\in(0,1)$, there hold
\begin{align*}
   &\lim_{\delta\to 0^+}\Prob\Bigg(k(\delta)\leq M\Bigg\lceil M^{-1}\bigg\lceil \frac{\tfrac{9}{8}n  \epsilon_0^2+3\sqrt{n}\|w_{\epsilon_0}\|(\tau-1) \delta \ln(\delta^{-1}+2)}{c_0(\tau-1)^2 \delta^2}\bigg\rceil\Bigg\rceil\Bigg)=1,\\
&\lim_{\delta\to 0^+}\E[\|x_{k(\delta)}^\delta-x^\dag\|^2|Q_{p_0}]^\frac12\leq  \big(\sqrt{2}+3(\tau-1)^{-1}\big)\epsilon_0,
\end{align*}
for some event $Q_{p_0}$ satisfying $\Prob(Q_{p_0})\geq 1-p_0$.
\end{itemize}
\end{theorem}

\section{Convergence analysis}\label{sec:conv}

Now we establish Theorems \ref{thm:dp_E} and \ref{thm:dp}, i.e., the finite-iteration termination property and convergence rates of SVRG equipped with the discrepancy principle \eqref{eqn:discrepancy_2}, and Theorem \ref{thm:regularizing}, i.e., the regularizing property.

\subsection{Notation and key preliminary estimates}\label{ssec:notation}

First we introduce several shorthand notation. 
Below, we denote the SVRG iterates for the noisy data $y^\delta$ by $x_k^\delta$. 
For any $k\in \mathbb{N}$, we define 
\begin{align}
&e_{k}^\delta=x_{k}^\delta-x^\dag,
\quad \Delta_{k}^\delta=x_{k}^\delta-x_{[k/M]M}^\delta, \quad r_k^\delta = Ax_k^\delta-y^\delta, \label{eqn:rk}\\
&N_{k}=B-A_{i_k}^*A_{i_k}, \quad P=I-c_0B
\quad\mbox{and}\quad q_k^\delta=c_0\sum_{j=1}^{k-1} P^{k-1-j} N_j\Delta_j^\delta,\label{eqn:qk}
\end{align}
with $B:=\E[A_i^*A_i]=n^{-1}A^* A\;: X\rightarrow X$, and $I$ being the identity operator.
Note that $\E[N_{k}]=0$ and $\Delta_{KM}^\delta=0$ for any $K\in \mathbb{N}$.
We also follow the convention 
$\sum_{j=i}^{i'} R_j=0$ for any sequence $\{R_j\}_j$ and $i'< i$.

Next we derive several key estimates. In the analysis below we use technical estimates in Appendix \ref{app:prelim}.
We first bound the residual $r_{k}^\delta=Ax_{k}^\delta-y^\delta$ in terms of the weighted successive errors $\E[\|A\Delta_j^\delta\|^2]$ and $\|A\Delta_{j}^\delta\|$, respectively. 

\begin{lemma}\label{lem:res}
Let Assumption \ref{ass}{\rm(i)} hold. Then for any $k\geq 0$, the following estimates hold
\begin{align}
\E[\|r_{k}^\delta\|^2]^\frac12\leq &\|\E[r_{k}^\delta]\|+\E[\|r_{k}^\delta-\E[r_{k}^\delta]\|^2]^\frac12\nonumber\\ 
\leq &(\|AP^{k}e_{0}^\delta\| +\delta) +\sqrt{n} \bigg(\E[\|A\Delta_{k-1}^\delta\|^2]+\sum_{j=1}^{k-2}(k-1-j)^{-2}\E[\|A\Delta_j^\delta\|^2]\bigg)^\frac12,\label{eqn:estim-res-1}\\ 
\|r_{k}^\delta\|\leq &\|AP^{k}e_{0}^\delta\| +\delta+n \bigg(\|A\Delta_{k-1}^\delta\|+\sum_{j=1}^{k-2}(k-1-j)^{-1}\|A\Delta_j^\delta\|\bigg).\label{eqn:estim-res-2}
\end{align}
\end{lemma}
\begin{proof}
By the splitting of the error $e_k^\delta$ in Lemma \ref{lem:bias-var}, we can decompose the residual $$r_k^\delta=Ax_{k}^\delta-A x^\dag-\xi=Ae_{k}^\delta-\xi$$ into
$
r_{k}^\delta=\E[r_k^\delta]+(r_{k}^\delta-\E[r_{k}^\delta]),
$
with the bias $\E[r_k^\delta]$ and the stochastic component $r_k^\delta-\E[r_k^\delta]$ given respectively by 
\begin{align}\label{eqn:bias-var_res}
\E[r_k^\delta]=&AP^{k}e_{0}^\delta +n^{-1}c_0A \sum_{j=0}^{k-1}P^{j}A^* \xi-\xi \quad \mbox{and} \quad  r_{k}^\delta-\E[r_{k}^\delta] = c_0A\sum_{j=1}^{k-1}P^{k-1-j}N_{j}\Delta_{j}^\delta.
\end{align} 
By the triangle inequality, the bias component $\E[r_{k}^\delta]$ is bounded by
\begin{align*}
\|\E[r_{k}^\delta]\|
\leq& \|AP^{k}e_{0}^\delta\| +\bigg\|n^{-1}c_0A \sum_{j=0}^{k-1}P^{j}A^* -I\bigg\|\delta.
\end{align*}
Note that 
\begin{align}
\bigg\|n^{-1}c_0A \sum_{j=0}^{k-1}P^{j}A^* -I\bigg\|=\bigg\|c_0B \sum_{j=0}^{k-1}P^{j} -I\bigg\|=\|(I-P^{k})-I\|\leq 1.\label{eqn:1delta}
\end{align}
Under Assumption \ref{ass}(i), by the identity $$\E[( AP^{k-1-i}N_{i}\Delta_{i}^\delta,AP^{k-1-j}N_{j}\Delta_{j}^\delta)|\mathcal{F}_i]=0$$ for any $i>j$, and Lemmas \ref{lem:kernel} and \ref{lem:N}, we have 
\begin{align}
&\E[\|r_{k}^\delta-\E[r_{k}^\delta]\|^2]=c_0^2\sum_{j=1}^{k-1}\E[\|AP^{k-1-j}N_{j}\Delta_{j}^\delta\|^2]\nonumber\\
\leq& c_0^2\sum_{j=1}^{k-1}\|AP^{k-1-j}B^\frac12\|^2\E[\|A\Delta_j^\delta\|^2]
\leq n c_0^2\sum_{j=1}^{k-1}\| P^{k-1-j}B\|^2\E[\|A\Delta_j^\delta\|^2]\nonumber\\
\leq& nc_0^2\|B\|^2\E[\|A\Delta_{k-1}^\delta\|^2]+ n \sum_{j=1}^{k-2}(k-1-j)^{-2}\E[\|A\Delta_j^\delta\|^2]\nonumber\\
\leq& n\bigg(\E[\|A\Delta_{k-1}^\delta\|^2]+ \sum_{j=1}^{k-2}(k-1-j)^{-2}\E[\|A\Delta_j^\delta\|^2]\bigg).\label{eqn:res-var}
\end{align}
Then, the triangle inequality 
$\E[\|r_{k}^\delta\|^2]^\frac12\leq \|\E[r_{k}^\delta]\|+\E[\|r_{k}^\delta-\E[r_{k}^\delta]\|^2]^\frac12$
completes the proof of the estimate \eqref{eqn:estim-res-1}.
Similarly, by Lemmas \ref{lem:kernel}, \ref{lem:N} and \ref{lem:bias-var}, we have
\begin{align}
\|r_{k}^\delta-\E[r_{k}^\delta]\|\leq& c_0\sum_{j=1}^{k-1}\|AP^{k-1-j}N_{j}\Delta_{j}^\delta\|
\leq \sqrt{n}c_0\sum_{j=1}^{k-1}\|AP^{k-1-j}B^\frac12\|\|A\Delta_j^\delta\|\nonumber\\
\leq &n c_0\sum_{j=1}^{k-1}\| P^{k-1-j}B\|\|A\Delta_j^\delta\|
\leq n \bigg( \|A\Delta_{k-1}^\delta\|+ \sum_{j=1}^{k-2}(k-1-j)^{-1}\|A\Delta_j^\delta\|\bigg).\label{eqn:res-var3}
\end{align}
Then the decomposition \eqref{eqn:bias-var_res} and the triangle inequality yield the estimate \eqref{eqn:estim-res-2}.
\end{proof}

Next, we bound the residuals $\E[\|r_{k}^\delta\|^2]^\frac12$ and $\|r_{k}^\delta\|$ in terms of the iteration index $k$.

\begin{theorem}\label{thm:res}
There exists $c^{*}>0$ independent of $w$, $k$, $n$, $M$ and $\delta$ such that, for any $k\geq 1$, the following statements hold.
\begin{enumerate}
\item[{\rm(i)}] If Assumption \ref{ass}{\rm(i)} holds, then 
\begin{align*}
\E[\|A(x_{k}^\delta-\E[x_{k}^\delta])\|^2]^\frac12&=\E[\|r_{k}^\delta-\E[r_{k}^\delta]\|^2]^\frac12\leq  \sqrt{n}Mc^{*}k^{-1}, \quad c_0<\overline{C_0}, \\
\|A(x_{k}^\delta-\E[x_{k}^\delta])\|&=\|r_{k}^\delta-\E[r_{k}^\delta]\|\leq nMc^{*}k^{-1}\ln k, \quad c_0<C_0.
\end{align*}
\item[{\rm(ii)}] If Assumption \ref{ass} holds, then with $c_{\nu}=(\nu+\frac12)^{\nu+\frac12}c_0^{-(\nu+\frac12)}$,
\begin{align*}
\|\E[r_{k}^\delta]\|\leq \sqrt{n}c_{\nu}k^{-(\nu+\frac12)}\|w\|+\delta,
\end{align*}
and with the convention $k^0:=\ln k$,
\begin{align*}
\E[\|r_k^\delta\|^2]^\frac12&\leq \sqrt{n}(c_\nu\|w\|+Mc^{*})k^{-\min(\nu+\frac12,1)}+\delta,   \quad c_0<\overline{C_0}, \\
\|r_k^\delta\|&\leq n (c_\nu\|w\|+Mc^{*})k^{-1}k^{\max(\frac12-\nu,0)}+\delta, \quad c_0<C_0. 
\end{align*}
\end{enumerate}
\end{theorem}
\begin{proof}
It follows from Lemma \ref{lem:Delta} (for $c_0<\overline{C_0}$), the estimate \eqref{eqn:res-var} in the proof of Lemma \ref{lem:res} and the assumption $\delta\leq 1$
that, for any $k\geq 2$, 
\begin{align*}
\E[\|r_{k}^\delta-\E[r_{k}^\delta]\|^2]^\frac12
\leq M\sqrt{n (c_1+c_2)\bigg[(k-1+M)^{-2}+\sum_{j=1}^{k-2}(k-1-j)^{-2}(j+M)^{-2}\bigg]}.
\end{align*}
The estimate \eqref{eqn:sum_2_2} in Lemma \ref{lem:sums} with $k-1\geq 1$ implies \begin{equation}\label{eqn:res-var2} 
\E[\|r_{k}^\delta-\E[r_{k}^\delta]\|^2]^\frac12\leq M\sqrt{8n (c_1+c_2)}k^{-1}.
\end{equation}
This estimate also holds for $k=1$ since $r^\delta_1=\E[r^\delta_1]$.
Assumption \ref{ass}(ii) and Lemma \ref{lem:kernel} indicate
\begin{align}\label{eqn:approx_res}
\|AP^{k}e_{0}^\delta\|
=&\sqrt{n} \|B^\frac12 P^{k}B^\nu w\|
\leq \sqrt{n} \|P^{k}B^{\nu+\frac12}\|\|w\|\leq\sqrt{n}c_{\nu}k^{-(\nu+\frac12)}\|w\|,
\end{align}
with the constant $c_{\nu}=(\nu+\frac12)^{\nu+\frac12}c_0^{-(\nu+\frac12)}$, and thus by Lemma \ref{lem:res}
\begin{align}
\|\E[r_{k}^\delta]\|\leq \|AP^{k}e_{0}^\delta\| +\delta\leq \sqrt{n}c_{\nu}k^{-(\nu+\frac12)}\|w\|+\delta.\label{eqn:res0_E}
\end{align}
Then, combining the estimates \eqref{eqn:res-var2} and \eqref{eqn:res0_E} gives the desired estimates when $c_0<\overline{C_0}$. Similarly, Lemma \ref{lem:Delta} (when $c_0<C_0$) and the estimate \eqref{eqn:res-var3} in the proof of Lemma \ref{lem:res} imply
\begin{align}\label{eqn:res0_as}
\|r_{k}^\delta\|
&\leq \|AP^{k}e_{0}^\delta\| +\delta+\|r_{k}^\delta-\E[r_{k}^\delta]\|,\\  \|r_{k}^\delta-\E[r_{k}^\delta]\|&\leq nM (c_1+c_2)\bigg[ (k-1+M)^{-1}+\sum_{j=1}^{k-2}(k-1-j)^{-1}(j+M)^{-1}\bigg],
\end{align} 
Note that $r_{1}^\delta=\E[r_{1}^\delta]$. 
For any $k-1\geq 1$, with the estimate \eqref{eqn:sum_1_1} in Lemma \ref{lem:sums}, we deduce $$\|r_{k}^\delta-\E[r_{k}^\delta]\|\leq 6 (c_1+c_2)nM k^{-1}\ln k.$$
These two estimates and the inequality \eqref{eqn:approx_res} yield the desired assertion (when $c_0<C_0$).
\end{proof}

Note that the stochastic gradient employed in SVRG is unbiased. Thus, the monotonically decreasing property of the residual in the Landweber method holds for the expected residual of SVRG; see the next lemma.
\begin{lemma}\label{lem:mono_res}
Let Assumption \ref{ass}{\rm(i)} be fulfilled. Then for any $k\geq 0$, there holds
$$\|\E[r_{k+1}^\delta]\|\leq\|\E[r_{k}^\delta]\|.$$
\end{lemma}
\begin{proof}
By the definition in Algorithm \ref{alg:svrg}, $\E[e_{k+1}^\delta] =\E[e_{k}^\delta]-c_0n^{-1}A^*\E[r_{k}^\delta]$. Consequently, 
\begin{align*}
\|\E[r_{k+1}^\delta]\|^2-\|\E[r_k^\delta]\|^2= &\|\E[r_{k+1}^\delta]-\E[r_k^\delta]\|^2+2( \E[r_{k+1}^\delta]-\E[r_k^\delta], \E[r_k^\delta])\\
= &\|A(\E[e_{k+1}^\delta]-\E[e_k^\delta])\|^2+2(e A(\E[e_{k+1}^\delta]-\E[e_k^\delta]), \E[r_k^\delta]) \\
= &c_0^2n^{-2}\|AA^*\E[r_{k}^\delta]\|^2-2c_0n^{-1}( AA^*\E[r_{k}^\delta], \E[r_k^\delta])\\
\leq &c_0^2n^{-1}L\|A^*\E[r_{k}^\delta]\|^2-2c_0n^{-1}\| A^*\E[r_{k}^\delta]\|^2\\
\leq& -c_0n^{-1}(2-c_0L)\| A^*\E[r_{k}^\delta]\|^2\leq 0.
\end{align*}
This completes the proof of the lemma.
\end{proof}

\subsection{Proof of Theorem \ref{thm:dp_E}}
Now we can prove Theorem \ref{thm:dp_E}, i.e., the finite-iteration termination property in probability and convergence rates in conditional expectation on an event $Q$ of high probability for SVRG. We first define the event 
\begin{align}\label{eqn:Q}
Q=\{\|B^{-\nu}N_j\Delta_j^\delta\|\leq j^{1-\frac1s} \rho, \;\forall j\geq \overline{k}(\delta)^{(1+2\nu)s}\},    
\end{align}
where the exponent $s$ and the constant $\rho$ are given by 
\begin{align*} 
s&\geq \max\Big(1,\frac13\Big(1+4\frac{\ln M}{\ln n}\Big)\Big),\\ \rho&=\sqrt{2(c_1+c_2\delta^2)}\|A\|^{1-2\nu}\left(\frac{\tau-1}{2c_\nu \|w\|}\right)^{3s-2}n^{-\frac{3s-1-2\nu}{2}}M,
\end{align*}
and the integer $\overline{k}(\delta)$ is defined by 
(with the convention $(\delta^{-1})^0:=\ln (\delta^{-1})$) \begin{equation}\label{eqn:def_bar_k}
\overline{k}(\delta):=M\Bigg\lceil M^{-1}\bigg\lceil\bigg(\frac{2\sqrt{n}c_\nu\|w\|}{\tau-1}\bigg)^{\frac{2}{1+2\nu}}\delta^{-1}\big(\delta^{-1}\big)^{\max(\frac{2}{1+2\nu}-1,0)}\bigg\rceil\Bigg\rceil.    
\end{equation}
Then by  Proposition \ref{prop:Q} in the appendix, there holds 
\begin{equation*} 
\Prob(Q)\geq \tfrac12\quad\mbox{and}\quad \lim_{\delta\to 0^+}\Prob(Q)=1;
\end{equation*}
The proof proceeds as follows. We first establish the finite-iteration termination property, and then decompose the conditional expectation $\E[\|e_{k(\delta)}^\delta\|^2|Q]^\frac12$ into noise propagation error and the remaining terms, and bound them in terms of the noise level $\delta$ via the interpolation inequality.

\begin{proof}[Proof of Theorem \ref{thm:dp_E}] The proof is broken down into three steps.\\
\underline{Step 1. Finite iteration termination property.}
Following the analysis in the proof of \cite[Theorem 1.1]{JahnJin:2020}, we define 
\begin{align*}
k_0(\delta)=\min\{k\in \mathbb{N}:\;\|\E[r_{k}^\delta]\|\leq \tfrac12(\tau+1) \delta\}.
\end{align*}
The estimate on $\|\E[r_{k}^\delta]\|$ in Theorem \ref{thm:res} gives
\begin{align*}
k_0(\delta)\leq\Bigg\lceil\bigg(\frac{2\sqrt{n}c_\nu \|w\|}{(\tau-1) \delta}\bigg)^{\frac{2}{1+2\nu}}\Bigg\rceil.
\end{align*}
Note that by the definition \eqref{eqn:def_bar_k} of $\overline{k}(\delta)$, we have $\overline{k}(\delta)\geq k_0(\delta)$.
By the monotonicity of the expected residual $\|\E[r_t^\delta]\|$ in Lemma \ref{lem:mono_res} and the definition of $k_0(\delta)$, for any $t\geq k_0(\delta)$, there holds
\begin{align*}
\|\E[r^\delta_{t}]\|\leq
\|\E[r^\delta_{k_0(\delta)}]\|
\leq \tfrac12(\tau+1) \delta.
\end{align*}
By the definition \eqref{eqn:def_bar_k} of $\overline{k}(\delta)$, we have $M\lfloor M^{-1}\overline{k}(\delta)\rfloor = \bar{k}(\delta)$. Then by Lemma \ref{lem:Chebyshev} with $\tau'=\frac12 (\tau+1)$, there holds with $c_{\tau,n,M}=4(\tau-1)^{-2}n M^2(c^{*})^2$ that
\begin{align}
&\Prob\Big(k(\delta)\leq \overline{k}(\delta)\Big)\geq 1-c_{\tau,n,M} \delta^{-2}\overline{k}(\delta)^{-2}\nonumber\\
\geq &1-c_{\tau,n,M} \bigg(\frac{\tau-1}{2\sqrt{n}c_\nu\|w\|}\bigg)^{\frac{4}{1+2\nu}}\bigg(\big(\delta^{-1}\big)^{\max(\frac{2}{1+2\nu}-1,0)}\bigg)^{-2}\xrightarrow{\;\;\;\delta \to 0^+\;\;} 1, \nonumber
\end{align}
which implies the finite-iteration termination property:
\begin{align}\label{eqn:k_bar}
\lim_{\delta\to 0^+}\Prob\Big(k(\delta)\leq \overline{k}(\delta)\Big)=1.
\end{align}
\underline{Step 2: Preliminary decomposition of the error.}
For any $\nu\in(0,\frac12]$, using Lemma \ref{lem:bias-var} and the identity $e_{k}^\delta=\E[e_{k}^\delta]+(e_{k}^\delta-\E[e_{k}^\delta])$ for any fixed $k$, with the shorthand notation $q_k^\delta$ in \eqref{eqn:qk}, we derive
\begin{align}
e_{k(\delta)}^\delta
=P^{k(\delta)}e_0^\delta+q_{k(\delta)}^\delta+n^{-1}c_0\sum_{j=0}^{k(\delta)-1}P^j A^*\xi
:=z_{k(\delta)}^\delta+n^{-1}c_0\sum_{j=0}^{k(\delta)-1}P^j A^* \xi.\label{eqn:decom_dp}
\end{align}
By the triangle inequality and Lemma \ref{lem:kernel}, we derive 
\begin{align}
\bigg\|n^{-1}c_0\sum_{j=0}^{k(\delta)-1}P^j A^* \xi\bigg\|&\leq \sqrt{\frac{c_0^2}{n}}\bigg\|\sum_{j=0}^{k(\delta)-1}P^j B^{\frac12}\bigg\| \delta
\leq \sqrt{\frac{c_0^2}{n}}\sum_{j=0}^{k(\delta)-1}\|P^j B^{\frac12}\|\delta\nonumber\\
&\leq \sqrt{\frac{c_0}{2n}}\sum_{j=0}^{k(\delta)-1}j^{-\frac12}\delta
\leq \sqrt{\frac{2c_0}{n}}\sqrt{k(\delta)}\delta.\label{eqn:noise0}
\end{align} 
Note that the stopping index $k(\delta)=k(\delta,\theta)$ with the random realization $\theta$ is measurable with respect to the filtration $\mathcal{F}$. 
Then by Proposition \ref{prop:Q}, the triangle inequality, the interpolation inequality in Lemma \ref{lem:moment} and Jensen's inequality, we can bound the conditional expectation $\E[\|e_{k(\delta)}^\delta\|^2|Q]^\frac12$ by 
\begin{align*}
\E[\|e_{k(\delta)}^\delta\|^2|Q]^\frac12\leq& \E[\|z_{k(\delta)}^\delta\|^2|Q]^\frac12+\E\bigg[\bigg\|n^{-1}c_0\sum_{j=0}^{k(\delta)-1}P^j A^* \xi\bigg\|^2\bigg|Q\bigg]^\frac12\\
\leq& \E[\|B^\frac12 z_{k(\delta)}^\delta\|^\frac{4\nu}{1+2\nu}\|B^{-\nu}z_{k(\delta)}^\delta\|^\frac{2}{1+2\nu}|Q]^\frac12+\sqrt{2c_0 n^{-1}\E[k(\delta)|Q]}\delta.
\end{align*}
By the defining identities $B=n^{-1}A^*A$ and $r^\delta_{k(\delta)}=Ae_{k(\delta)}^\delta-\xi$, the decomposition \eqref{eqn:decom_dp}, the triangle inequality and the inequality \eqref{eqn:1delta}, we can bound the term $\|B^\frac12 z_{k(\delta)}^\delta\|$ by
\begin{align*}
\sqrt{n}\|B^\frac12 z_{k(\delta)}^\delta\|=&\|A z_{k(\delta)}^\delta\|= \bigg\|Ae_{k(\delta)}^\delta-n^{-1}c_0A\sum_{j=0}^{k(\delta)-1}P^j A^* \xi\bigg\|\\
\leq& \|r_{k(\delta)}^\delta\|+\bigg\|I-n^{-1}c_0A\sum_{j=0}^{k(\delta)-1}P^j A^* \bigg\|\delta
\leq \|r_{k(\delta)}^\delta\|+\delta.
\end{align*}
Then, the discrepancy principle \eqref{eqn:discrepancy_2} gives 
\begin{align}
\sqrt{n}\|B^\frac12 z_{k(\delta)}^\delta\|
\leq (\tau+1)\delta.\label{eqn:z}
\end{align}
\underline{Step 3. Bounds on the two terms.}
By Jensen's inequality, we obtain 
\begin{align}\label{eqn:err_E_dp}
\E[\|e_{k(\delta)}^\delta\|^2|Q]^\frac12\leq (n^{-\frac12}(\tau+1)\delta)^\frac{2\nu}{1+2\nu}\E[\|B^{-\nu}z_{k(\delta)}^\delta\|^2|Q]^\frac{1}{2+4\nu}+\sqrt{2c_0 n^{-1}\E[k(\delta)|Q]}\delta:= \mathfrak{I} + \mathfrak{K}
\end{align}
Next we bound the two terms $\mathfrak{I}$ and $\mathfrak{K}$ separately.
Then by Lemma \ref{lem:z_k}, we have 
\begin{align*}
\E[\|B^{-\nu}z_{k(\delta)}^\delta\|^2|Q]^\frac12\leq \overline{C}_{w,\nu,\tau} \max\big(1,\tfrac{\tau-1}{2c_\nu \|w\|}\big)^{3s}s n^\nu M\ln\big(\delta^{-1}+M\big),
\end{align*}
where the constant $\overline{C}_{w,\nu,\tau}$ is independent of $s$, $n$, $M$ and $\delta$. This implies
\begin{align}
\mathfrak{I}\leq& \overline{C}_{w,\nu,\tau}^\frac{1}{1+2\nu}\big(n^{-\frac12}(\tau+1)\delta\big)^\frac{2\nu}{1+2\nu} \big( \max\big(1,\tfrac{\tau-1}{2c_\nu \|w\|}\big)^{3s}s n^\nu
M\big)^\frac{1}{1+2\nu} \ln^\frac{1}{1+2\nu} (\delta^{-1}+M)\nonumber\\
\leq& \overline{C}_{w,\nu,\tau}^\frac{1}{1+2\nu}(\tau+1)^\frac{2\nu}{1+2\nu} \big(s\max\big(1,\tfrac{\tau-1}{2c_\nu \|w\|}\big)^{3s}\big)^\frac{1}{1+2\nu} 
M^\frac{1}{1+2\nu} \delta^\frac{2\nu}{1+2\nu} \ln^\frac{1}{1+2\nu} (\delta^{-1}+M)\label{eqn:err_E_dp_2}.
\end{align}
Finally, the estimates \eqref{eqn:Ek} with $\tau'=\frac12(\tau+1)$, \eqref{eqn:bar_k_bd} and \eqref{eqn:bark-2} and the elementary inequality 
$$\delta^{-1}\big(\delta^{-1}\big)^{\max(\frac{2}{1+2\nu}-1,0)}\leq \delta^{-\frac{2}{1+2\nu}}\ln\big(\delta^{-1}+M\big)$$ yield that for any $\nu\in(0,\frac12]$, there holds
\begin{align*}
\mathfrak{K}\leq & \sqrt{2c_0 n^{-1}\Prob(Q)^{-1}\E[k(\delta)]}\delta\leq 2\sqrt{c_0 n^{-1}\big(2\bar{k}(\delta)+c_{\tau,n,M}\delta^{-2}\bar{k}(\delta)^{-1}\big)}\delta\\
\leq& 2\sqrt{c_0 n^{-1}}\Bigg(2c_{w,\nu,\tau}\big( n^{\frac{1}{1+2\nu}}\delta^{-1}\big(\delta^{-1}\big)^{\max(\frac{2}{1+2\nu}-1,0)}+M\big)\\
&\qquad\qquad\quad+4 (2c_\nu \|w\|)^{-\frac{2}{1+2\nu}} (\tau-1)^{-\frac{4\nu}{1+2\nu}}(c^{*})^2n^{\frac{2\nu}{1+2\nu}}M^2\delta^{-\frac{4\nu}{1+2\nu}}\Bigg)^\frac12 \delta \\
\leq& \overline{C}'_{w,\nu,\tau} \max\big(1,n^{-\frac{1}{2+4\nu}} M^{\frac{2\nu}{1+2\nu}}\big)M^\frac{1}{1+2\nu} \delta^{-\frac{1}{1+2\nu}}\sqrt{\ln\big(\delta^{-1}+M\big)}\delta \\
\leq& \overline{C}'_{w,\nu,\tau} \max\big(1,n^{-\frac{1}{2+4\nu}} M^{\frac{2\nu}{1+2\nu}}\big)M^\frac{1}{1+2\nu} \delta^\frac{2\nu}{1+2\nu}\ln^\frac{1}{1+2\nu} (\delta^{-1}+M).
\end{align*}
Hence, the inequality \eqref{eqn:err_E_dp} and the bounds on $\mathfrak{I}$ and $\mathfrak{K}$ imply 
\begin{align*}
\E[\|e_{k(\delta)}^\delta\|^2|Q]^\frac12\leq& \overline{C}_{w,\nu,\tau}^\frac{1}{1+2\nu}(\tau+1)^\frac{2\nu}{1+2\nu} \big(s\max\big(1,\tfrac{\tau-1}{2c_\nu \|w\|}\big)^{3s}\big)^\frac{1}{1+2\nu}
M^\frac{1}{1+2\nu} \delta^\frac{2\nu}{1+2\nu} \ln^\frac{1}{1+2\nu} (\delta^{-1}+M)\\
&+\overline{C}'_{w,\nu,\tau} \max\big(1,n^{-\frac{1}{2+4\nu}} M^{\frac{2\nu}{1+2\nu}}\big)M^\frac{1}{1+2\nu} \delta^\frac{2\nu}{1+2\nu}\ln^\frac{1}{1+2\nu}(\delta^{-1}+M)\\
\leq&  \overline{C}_{w,\nu,\tau}^* \max\big(s,s\big(\tfrac{\tau-1}{2c_\nu \|w\|}\big)^{3s},n^{-\frac{1}{2}}M^{2\nu}\big)^\frac{1}{1+2\nu} M^\frac{1}{1+2\nu} \delta^\frac{2\nu}{1+2\nu}\ln^\frac{1}{1+2\nu}(\delta^{-1}+M).
\end{align*}
This completes the proof of the theorem.
\end{proof}

\subsection{Proof of Theorem \ref{thm:dp}}

Next we show the second main result, i.e., the finite-iteration termination property and convergence rate for SVRG in the uniform sense in Theorem \ref{thm:dp}.

\begin{proof}[Proof of Theorem \ref{thm:dp}]
The proof strategy is similar to that of Theorem \ref{thm:dp_E}. First, 
Theorem \ref{thm:res}(ii) (when $c_0<C_0$) with the convention $k^0:=\ln k$ gives
\begin{align*}
\|r_k^\delta\|\leq \delta+n (c_\nu\|w\|+c^{*})
\left\{\begin{array}{cc}
k^{-(\nu+\frac12)},     & \nu<\frac12 ,\\
k^{-1}\ln k,     & \nu\geq \frac12.
\end{array}\right.
\end{align*}
For any $\epsilon\in(0,\frac12)$, the elementary inequality $k^{-\epsilon}\ln k\leq (e\epsilon)^{-1}$ yields 
\begin{align}\label{eqn:lnk}
k^{-1}\ln k\leq (e\epsilon)^{-1}k^{-(1-\epsilon)}.   
\end{align}
Then, by the discrepancy principle \eqref{eqn:discrepancy_2}, we have (with the shorthand notation $c_{w,\nu,\tau}=\frac{c_\nu\|w\|+c^{*}}{\tau-1}$)
\begin{align*}
k(\delta)\leq \left\{\begin{array}{cc}
M\big\lceil M^{-1}\big((c_{w,\nu,\tau}n\delta^{-1})^{\frac{2}{1+2\nu}}+1\big)\big\rceil,     & \nu<\frac12, \\
M\big\lceil M^{-1}\big((c_{w,\nu,\tau}n(e\epsilon\delta)^{-1})^{\frac{1}{1-\epsilon}}+1\big)\big\rceil,     & \nu\geq \frac12.
\end{array}\right.
\end{align*}
For any $\nu\in(0,\frac12)$, by the identity \eqref{eqn:decom_dp} and the estimates \eqref{eqn:noise0} and \eqref{eqn:z}, we derive  
\begin{align}
\|e_{k(\delta)}^\delta\|\leq &\|z_{k(\delta)}^\delta\|+\bigg\|n^{-1}c_0\sum_{j=0}^{k(\delta)-1}P^j A^* \xi\bigg\|\nonumber\\
\leq& \|B^\frac12 z_{k(\delta)}^\delta\|^\frac{2\nu}{1+2\nu}\|B^{-\nu}z_{k(\delta)}^\delta\|^\frac{1}{1+2\nu}+\sqrt{2n^{-1}c_0}\sqrt{k(\delta)} \delta\nonumber\\
\leq& (n^{-\frac12}(\tau+1)\delta)^\frac{2\nu}{1+2\nu}\|B^{-\nu}z_{k(\delta)}^\delta\|^\frac{1}{1+2\nu}+\sqrt{2n^{-1}c_0}\sqrt{k(\delta)} \delta.
\label{eqn:err_as_dp}
\end{align}
Next, under Assumption \ref{ass}(ii), by Lemma \ref{lem:N}, we have
\begin{align*}
\|B^{-\nu}z_{k(\delta)}^\delta\|\leq& \|B^{-\nu}P^{k(\delta)}e_0^\delta\|+c_0\sum_{j=1}^{k(\delta)-1}\|B^{-\nu} P^{k(\delta)-1-j} N_j\Delta_j^\delta\|\\
\leq &\|P^{k(\delta)}w\|+\sqrt{n}c_0\|B^{\frac12-\nu}\|\|A\Delta_{k(\delta)-1}^\delta\|+\sqrt{n}c_0\sum_{j=1}^{k(\delta)-2}\|P^{k(\delta)-1-j} B^{\frac12-\nu}\|\|A\Delta_j^\delta\|.
\end{align*}
Further, by Lemmas \ref{lem:kernel} and \ref{lem:Delta}, the estimate \eqref{eqn:sum_nu_1} in Lemma \ref{lem:sums} with $k=k(\delta)-1$ and the assumption $\delta\leq 1$, there hold
\begin{align*}
\|B^{-\nu}z_{k(\delta)}^\delta\|
\leq& \|w\|+\sqrt{n}c_0 \|B\|^{\frac12-\nu}(c_1+c_2\delta)M(k(\delta)-1+M)^{-1}\\
&+(\tfrac12-\nu)^{\frac12-\nu}\sqrt{n}c_0^{\frac12+\nu}(c_1+c_2\delta)M\sum_{j=1}^{k(\delta)-2}(k(\delta)-1-j)^{-(\frac12-\nu)} (j+M)^{-1}\\
\leq&\sqrt{n}M C_{w,\nu}<\infty,
\end{align*}
where the constant $C_{w,\nu}$ is independent of $n$, $M$, $k(\delta)$ and  $\delta$.   
Thus, from the estimate \eqref{eqn:err_as_dp} and the assumption $\delta\leq 1$, we derive 
\begin{align*}
\|e_{k(\delta)}^\delta\|
\leq &\big(n^{-\frac12}(\tau+1)\delta\big)^\frac{2\nu}{1+2\nu}(n^{\frac12}MC_{w,\nu})^\frac{1}{1+2\nu}+\sqrt{2n^{-1}c_0\big(M\big\lceil M^{-1}\big((c_{w,\nu,\tau}n \delta^{-1})^{\frac{2}{1+2\nu}}+1\big)\big\rceil\big)} \delta\\
\leq& C_{w,\nu}^\frac{1}{1+2\nu}(\tau+1)^\frac{2\nu}{1+2\nu}n^\frac{1-2\nu}{2+4\nu}M^\frac{1}{1+2\nu} \delta^\frac{2\nu}{1+2\nu}+\sqrt{2c_0\big(c_{w,\nu,\tau}^{\frac{2}{1+2\nu}}n^{\frac{1-2\nu}{1+2\nu}}+Mn^{-1}\big)}  \delta^{\frac{2\nu}{1+2\nu}}\\
\leq& C_{w,\nu}^\frac{1}{1+2\nu}(\tau+1)^\frac{2\nu}{1+2\nu}n^\frac{1-2\nu}{2+4\nu}M^\frac{1}{1+2\nu} \delta^\frac{2\nu}{1+2\nu}+2\sqrt{c_0}\max\big(c_{w,\nu,\tau}^{\frac{1}{1+2\nu}}n^{\frac{1-2\nu}{2+4\nu}}, n^{-\frac12}M^{\frac12}\big)  \delta^{\frac{2\nu}{1+2\nu}}\\
\leq& C^*_{w,\nu,\tau}
\max\Big(n^\frac{1-2\nu}{2+4\nu}M^\frac{1}{1+2\nu},n^{-\frac12}M^\frac12\Big) 
\delta^{\frac{2\nu}{1+2\nu}}
\leq C^*_{w,\nu,\tau}
n^\frac{1-2\nu}{2+4\nu} M^\frac{1}{1+2\nu} 
\delta^{\frac{2\nu}{1+2\nu}}.
\end{align*}
This completes the proof of the theorem.
\end{proof}

\subsection{Proof of Theorem \ref{thm:regularizing}}
To establish the regularizing property of SVRG equipped with the discrepancy principle \eqref{eqn:discrepancy_2}, we need one preliminary estimate. 
The next corollary is direct from Theorem \ref{thm:res}(ii) under Assumption \ref{ass}(i). The definition of $x^\dag$ implies $x^\dag-x_0\in \overline{\mathrm{Range}\big(A^*\big)}$ and 
the polar decomposition $A=Q(A^* A)^\frac12$ with a partial isometry $Q$ (i.e. $Q^* Q$ and $Q Q^*$ are both projections) further implies $x^\dag-x_0\in \overline{\mathrm{Range}\big((A^* A)^\frac12\big)}$.
Thus, for any $\epsilon_0\in(0,1)$, there exists some  $\tilde{x}_0$ such that $\|x_0-\tilde{x}_0\|<\epsilon_0$ and Assumption \ref{ass}(ii) holds with $\nu=\frac12$ and some $w=w_{\epsilon_0}$.
\begin{corollary}\label{cor:res}
Let Assumption \ref{ass}{\rm(i)} hold. Then for any $k\geq 1$ and $\epsilon_0\in(0,1)$, there hold
\begin{align*}
\|P^{k}e_0^\delta\|
&\leq \epsilon_0+(2c_0)^{-\frac12} k^{-\frac12}\|w_{\epsilon_0}\|,\\
\|\E[r_{k}^\delta]\|&\leq \sqrt{n}(2c_0)^{-\frac12}k^{-\frac12}\epsilon_0+
\sqrt{n}c_0^{-1}k^{-1}\|w_{\epsilon_0}\|+\delta,\\
\|r_k^\delta\|&\leq \sqrt{n}(2c_0)^{-\frac12}k^{-\frac12}\epsilon_0+n (c_0^{-1}\|w_{\epsilon_0}\|+{M c^{*}})k^{-1}\ln k+\delta, \quad c_0<C_0,
\end{align*}
where $w_{\epsilon_0}$ depends only on $\epsilon_0$, and $c^{*}>0$ is independent of $\epsilon_0$, $k$, $n$, {$M$} and $\delta$.
\end{corollary}
\begin{proof} 
Let $\tilde{x}^\delta_k$ be the SVRG iterate starting with $\tilde{x}_0$ and  $\tilde{e}^\delta_k=\tilde{x}^\delta_k-x^\dag$. 
Then, by Lemma \ref{lem:kernel}, there hold
\begin{align*}
\|P^{k}e_0^\delta\|
\leq \|P^{k}(\tilde{e}_0^\delta-e_0^\delta)\|+\|P^{k}\tilde{e}_0^\delta\|
\leq \epsilon_0+\|P^{k}B^\frac12 w_{\epsilon_0}\|
\leq \epsilon_0+(2c_0)^{-\frac12}k^{-\frac12}\|w_{\epsilon_0}\|.
\end{align*}
Similarly, we can bound the term $\|AP^{k}e_{0}^\delta\|$ in \eqref{eqn:res0_E} and \eqref{eqn:res0_as} by
\begin{align*}
\|AP^{k}e_{0}^\delta\|\leq& \|AP^{k}(\tilde{e}_{0}^\delta-e_{0}^\delta)\|+\|AP^{k}\tilde{e}_{0}^\delta\|
\leq \sqrt{n}\|B^{\frac12}P^k\|\epsilon_0+\sqrt{n}\|P^{k}B w_{\epsilon_0}\|\\
\leq& \sqrt{n}(2c_0)^{-\frac12}k^{-\frac12}\epsilon_0+
\sqrt{n}c_0^{-1}k^{-1}\|w_{\epsilon_0}\|.
\end{align*}
This and \eqref{eqn:res0_E} yield the desired bound on $\|\E[r_{k}^\delta]\|$. Further, the decomposition $\|r_{k}^\delta\|\leq \|\E[r_{k}^\delta]\|+\|r_{k}^\delta-\E[r_{k}^\delta]\|$ and Theorem \ref{thm:res} (i) give the bound on $\|r_{k}^\delta\|$.
\end{proof}

Now we prove Theorem \ref{thm:regularizing} on the regularizing property of SVRG equipped with the discrepancy principle \eqref{eqn:discrepancy_2}.
The proof proceeds as follows.  (i) we first establish the finite-iteration termination property in the uniform sense using the estimate on the residual $\|r_k^\delta\|$. For the regularizing property, we consider two regimes separately, i.e., $\lim_{\delta\to 0^+}k(\delta)=\infty$ and $\lim_{\delta\to 0^+}k(\delta)<\infty$. When $\lim_{\delta\to 0^+}k(\delta)=\infty$, we decompose the error $\|e_{k(\delta)}^\delta\|$ into three terms and bound each part in terms of $k(\delta)$, $\delta$ and $\epsilon_0$. 
(ii) We then prove similar results with high probability.

\begin{proof}[Proof of Theorem  \ref{thm:regularizing}]
We analyze the uniform and high probability convergence separately.

\medskip
\noindent (i) First, we show the results when $c_0<C_0$. 
For any $\epsilon_0\in(0,1)$ and $\epsilon\in(0,\frac12)$, the inequality \eqref{eqn:lnk} and Corollary \ref{cor:res} yield
\begin{align*}
\|r_k^\delta\|&\leq \sqrt{n}(2c_0)^{-\frac12}k^{-\frac{1-\epsilon}{2}}\epsilon_0+(e\epsilon)^{-1}n (c_0^{-1}\|w_{\epsilon_0}\|+M c^{*})k^{-(1-\epsilon)}+\delta.
\end{align*}
Then by the definition of $k(\delta)$ in the discrepancy principle \eqref{eqn:discrepancy_2}, there holds 
\begin{align}\label{eqn:kt}
k(\delta)\leq M\Big\lceil M^{-1}\lceil t\rceil\Big\rceil
\leq t+M+1,
\end{align}
where $t$ satisfies
\begin{align*}
&(\tau-1) \delta t^{1-\epsilon}-\sqrt{\frac{n}{2c_0}}\epsilon_0\sqrt{t^{1-\epsilon}}-
(e\epsilon)^{-1}n(c_0^{-1}\|w_{\epsilon_0}\|+Mc^{*})= 0.
\end{align*}
Solving the above inequality for $k(\delta)$ yields the finite-iteration termination property:
\begin{align*}
k(\delta)\leq&\Bigg(\frac{\sqrt{\frac{n}{2c_0}}\epsilon_0+\sqrt{\frac{n}{2c_0}\epsilon_0^2+4(e\epsilon)^{-1}n(\tau-1)(c_0^{-1}\|w_{\epsilon_0}\|+Mc^{*}) \delta}}{2(\tau-1) \delta}\Bigg)^\frac{2}{1-\epsilon}+M+1\\
\leq&\Bigg(\frac{n  \epsilon_0^2+2\epsilon^{-1}c_0n(\tau-1)(c_0^{-1}\|w_{\epsilon_0}\|+Mc^{*}) \delta}{2c_0(\tau-1)^2 \delta^2}\Bigg)^\frac{1}{1-\epsilon}+M+1.
\end{align*}
Next, we bound the error $\|e_{k(\delta)}^\delta\|$ at the stopping index $k(\delta)$ for any fixed path. 
Let $\{\delta_t\}_{t=1}^\infty$ be any monotonically decreasing sequence such that $\lim_{t\to \infty}\delta_t=0$. Without loss of generality, we assume that $\lim_{t \to\infty}k(\delta_t)=\infty$ or that there exist some finite $t_*$ and $k_*<\infty$ such that $k(\delta_t)=k_*$ for any $t\geq t_*$.
In the second case, there holds 
$\|r_{k_*}^{\delta_t}\|\leq \tau \delta_t$.
By the triangle inequality, the definition \eqref{eqn:discrepancy_2} of the discrepancy principle and taking the limit as $t\to\infty$, we deduce
\begin{align*}
\|A\lim_{t\to \infty}e_{k_*}^{\delta_t}\|=\lim_{t\to \infty}\|r_{k_*}^{\delta_t}\|
+\delta_t\leq (\tau+1) \lim_{t\to \infty}\delta_t=0.
\end{align*}
Thus, $e_{k_*}^{\delta_t}-e_{0}^{\delta_t}\in \mathrm{Range}\big(A^*\big)$ and $e_{0}^{\delta_t}\in \overline{\mathrm{Range}\big(A^*\big)}$ imply
\begin{align}\label{eqn:err_finite}
\lim_{t\to \infty}\|e_{k_*}^{\delta_t}\|=0,\quad \mbox{i.e.,}\quad \lim_{\delta\to 0^+}\|e_{k(\delta)}^{\delta}\|=0.
\end{align}
In the first case, i.e., when $\lim_{t \to\infty}k(\delta_t)=\infty$, by the decomposition in Lemma \ref{lem:bias-var} and the triangle inequality, we obtain
\begin{align}\label{eqn:cor_decom}
\|e_{k(\delta)}^\delta\|
\leq& \|P^{k(\delta)}e_0^\delta\|+\bigg\|n^{-1}c_0\sum_{j=0}^{k(\delta)-1}P^j A^* \xi\bigg\|+\|q_{k(\delta)}^\delta\|.
\end{align}
By Corollary \ref{cor:res}, the discrepancy principle \eqref{eqn:discrepancy_2} and the estimate \eqref{eqn:kt}, for any $\epsilon_0\in(0,1)$, the stopping index $k(\delta)\leq t+M+1$ with $t$ satisfying 
\begin{align*}
\sqrt{n}(2c_0)^{-\frac12}t^{-\frac12}\epsilon_0+n {(c_0^{-1}\|w_{\epsilon_0}\|+Mc^{*})}t^{-1}\ln t+\delta=\tau \delta,
\end{align*}
which implies
\begin{align*}
(\tau-1) \sqrt{t} \delta=\sqrt{n}(2c_0)^{-\frac12}\epsilon_0+n (c_0^{-1}\|w_{\epsilon_0}\|+Mc^{*})t^{-\frac12}\ln t.
\end{align*}
This and the inequality \eqref{eqn:noise0} yield
\begin{align*}
&\quad{\bigg\|n^{-1}c_0\sum_{j=0}^{k(\delta)-1}P^j A^* \xi\bigg\|
\leq \sqrt{\frac{2c_0}{n}}\sqrt{k(\delta)} \delta
\leq \sqrt{\frac{2c_0}{n}}\sqrt{t+M+1} \delta}
\leq \sqrt{\frac{2c_0}{n}}\sqrt{t} \delta+\sqrt{\frac{2c_0(M+1)}{n}}\delta\\
&\leq \sqrt{\frac{2c_0}{n}}(\tau-1)^{-1}
\Big(\sqrt{n}(2c_0)^{-\frac12}\epsilon_0+n (c_0^{-1}\|w_{\epsilon_0}\|+Mc^{*})t^{-\frac12}\ln t\Big)+\sqrt{\frac{2c_0(M+1)}{n}}\delta\\
&=  (\tau-1)^{-1}
\Big(\epsilon_0+\sqrt{2c_0 n}(c_0^{-1}\|w_{\epsilon_0}\|+Mc^{*})t^{-\frac12}\ln t\Big)+\sqrt{2n^{-1}c_0(M+1)}\delta \\
&{\xrightarrow{\;\;\delta\to 0^+,\;\;t\geq k(\delta)-M-1 \to \infty\;\;} (\tau-1)^{-1}
\epsilon_0.}
\end{align*}
Then, for any $k(\delta)\geq 3$, Lemma \ref{lem:q_k_E} implies $\|q_{k(\delta)}^\delta\|\xrightarrow{\;\;k(\delta) \to \infty\;\;} 0$.
Thus, when $\lim_{t \to\infty}k(\delta_t)=\infty$, by the decomposition \eqref{eqn:cor_decom} and Corollary \ref{cor:res}, there holds
\begin{align*}
\lim_{t\to\infty}\|e_{k(\delta_t)}^{\delta_t}\|
\leq& \lim_{t\to\infty}\Big(\epsilon_0+\big(2c_0 k(\delta_t)\big)^{-\frac12}\|w_{\epsilon_0}\|+ (\tau-1)^{-1}
\epsilon_0\Big)=\tau(\tau-1)^{-1}\epsilon_0,
\end{align*}
i.e., 
$$\lim_{\delta\to 0^+}\|e_{k(\delta)}^\delta\|\leq \tau(\tau-1)^{-1}\epsilon_0.$$ 
This completes the proof of part (i) of the theorem.

\bigskip
\noindent (ii) The proof of this part is lengthy, and is divided into three steps. \\
\underline{Step 1. Finite iteration termination property.} When $c_0<\overline{C_0}$, following the argument of Theorem \ref{thm:dp_E}, we define 
\begin{align*}
k_0(\delta)=\min\{k\in \mathbb{N}:\;\|\E[r_{k}^\delta]\|\leq \tfrac13(2\tau+1) \delta\}.
\end{align*}
For any $\epsilon_0\in(0,1)$, Corollary \ref{cor:res} implies $k_0(\delta)\leq \lceil t\rceil$ with $t$ satisfying
\begin{align*}
&\tfrac23(\tau-1) \delta t-\sqrt{\frac{n}{2c_0}}\epsilon_0\sqrt{t}-
\frac{\sqrt{n}}{c_0}\|w_{\epsilon_0}\|= 0.
\end{align*}
Solving the equation for $t$ yields
\begin{align*}
k_0(\delta)\leq&\Bigg\lceil\Bigg(\frac{\sqrt{\frac{n}{2c_0}}\epsilon_0+\sqrt{\frac{n}{2c_0}\epsilon_0^2+\tfrac83(\tau-1) \delta \frac{\sqrt{n}}{c_0}\|w_{\epsilon_0}\|}}{\tfrac43(\tau-1) \delta}\Bigg)^2\Bigg\rceil
\leq\Bigg\lceil \frac{\tfrac{9}{8}n  \epsilon_0^2+3\sqrt{n}\|w_{\epsilon_0}\|(\tau-1) \delta }{c_0(\tau-1)^2 \delta^2}\Bigg\rceil.
\end{align*}
Then, when $\delta\leq1$, upon letting \begin{align}\label{eqn:bar_k_re}
\overline{k}(\delta):={M\Bigg\lceil M^{-1}\bigg\lceil \frac{\tfrac{9}{8}n  \epsilon_0^2+3\sqrt{n}\|w_{\epsilon_0}\|(\tau-1) \delta \ln(\delta^{-1}+2)}{c_0(\tau-1)^2 \delta^2}\bigg\rceil\Bigg\rceil}
\geq k_0(\delta),
\end{align}
Lemmas \ref{lem:mono_res} and \ref{lem:Chebyshev} with ($\tau'=\tfrac13(2\tau+1)$) yield
\begin{align*}
&\lim_{\delta\to 0^+}\Prob\Big(k(\delta)\leq \overline{k}(\delta)\Big) \geq 1-\lim_{\delta\to 0^+}(\tfrac{\tau-1}3)^{-2}n (c^{*})^2\delta^{-2}\big(M^{-1}\overline{k}(\delta)\big)^{-2}\\
\geq & 1-\lim_{\delta\to 0^+}9 n (c^{*})^2 \bigg(\frac{\tfrac{9}{8}n  \epsilon_0^2+3\sqrt{n}\|w_{\epsilon_0}\|(\tau-1) \delta \ln(\delta^{-1}+2)}{c_0(\tau-1)M \delta}\bigg)^{-2}
=1,
\end{align*}
i.e.,
$$\lim_{\delta\to 0^+}{\Prob\big(k(\delta)\leq \overline{k}(\delta)\big)}=1.$$

\medskip
\noindent\underline{Step 2. Basic error decomposition.} Now, we establish the regularizing property of SVRG in high probability. 
First we define the following two events 
\begin{align*} 
\widetilde{Q} &=\{\theta:\lim_{\delta\to 0^+}k(\delta,\theta)=\infty\},\\
\widehat{Q} & =  \Big\{\theta \in \tilde{Q}^c: \exists \;C_{\max}>0, \delta_0\in(0,1), \mbox{ s.t.}\sup_{\delta\in(0,\delta_0]}\|e^\delta_{k(\delta,\theta)}\|<C_{\max} \Big\},
\end{align*}
with the random realization $\theta$ and moreover, let 
\begin{align}\label{eqn:Q_p}
Q_{p_0}=\widetilde{Q}\cup\widehat{Q},    
\end{align}
for any $p_0\in(0,\frac12)$. Proposition \ref{prop:Q_p} gives $\Prob(Q_{p_0})\geq 1-p_0\geq \frac12$.
For the conditional mean squared error 
\begin{align*}
\E[\|e_{k(\delta)}^\delta\|^2| Q_{p_0}]^\frac12 \leq \Prob(Q_{p_0})^{-\frac12}\Big(\E\big[\|e_{k(\delta)}^\delta\|^2|\widetilde{Q}\big]\Prob(\widetilde{Q})+\E\big[\|e_{k(\delta)}^\delta\|^2|\widehat{Q}\big]\Big)^\frac12
\end{align*} 
at the stopping index $k(\delta)$, which is measurable with respect to the filtration $\mathcal{F}$, it follows from \eqref{eqn:err_finite} that
\begin{align*}
\lim_{\delta\to 0^+}\|e_{k(\delta)}^{\delta}\|&=0, \quad \mbox{when}\quad \lim_{\delta\to 0^+}k(\delta)<\infty,
\end{align*}
which implies $$\lim_{\delta\to 0^+}\E\big[\|e_{k(\delta)}^\delta\|^2|\widehat{Q}\big]=0.$$
Thus, 
\begin{align}
&\lim_{\delta\to 0^+}\E[\|e_{k(\delta)}^\delta\|^2| Q_{p_0}]^\frac12\leq 
\sqrt{2}\lim_{\delta\to 0^+}\E\big[\|e_{k(\delta)}^\delta\|^2|\widetilde{Q}\big]^\frac12\Prob(\widetilde{Q})^\frac12\nonumber\\
\leq & \sqrt{2}\Big(\epsilon_0+\sqrt{\lim_{\delta\to 0^+}2c_0 n^{-1}\E\big[k(\delta)|\widetilde{Q}\big] \Prob(\widetilde{Q})\delta^2}+\sqrt{\lim_{\delta\to 0^+}\E\big[\|q_{k(\delta)}^\delta\|^2|\widetilde{Q}\big]\Prob(\widetilde{Q})}\Big),\label{eqn:cor_e_E}
\end{align}
Next we bound the two limits in the bracket separately. By the estimate \eqref{eqn:Ek} with $\tau'=\tfrac13(2\tau+1)$, we derive
\begin{align*}
\E[k(\delta)]
\leq 2\overline{k}(\delta)+c'_{\tau,n,M}\delta^{-2}\bar{k}(\delta)^{-1},
\end{align*}
with $c'_{\tau,n,M}=9(\tau-1)^{-2}n M^2(c^{*})^2$. 
Then direct computation gives
\begin{align*}
&\lim_{\delta\to 0^+}2c_0 n^{-1}\E\big[k(\delta)|\widetilde{Q}\big] \Prob(\widetilde{Q})\delta^2
\leq 2c_0 n^{-1} \lim_{\delta\to 0^+}\E\big[k(\delta)\big]\delta^2\\
\leq& 2c_0 n^{-1} \lim_{\delta\to 0^+}\big(2\overline{k}(\delta)+c'_{\tau,n,M}\delta^{-2}\overline{k}(\delta)^{-1}\big)\delta^2\\
\leq& 2c_0 n^{-1} \big(2\lim_{\delta\to 0^+}\overline{k}(\delta)\delta^2+c'_{\tau,n,M}\lim_{\delta\to 0^+}\overline{k}(\delta)^{-1}\big).
\end{align*}
Meanwhile by the definition \eqref{eqn:bar_k_re} and the following inequality (derived from \eqref{eqn:ln})
$$\delta^{\frac12}\ln(\delta^{-1}+2)\leq \delta^{\frac12}\ln(3\delta^{-1})\leq \delta^{\frac12}\big(\ln(\delta^{-1})+\ln3\big)\leq 2e^{-1}+\ln 3<2,$$ we can bound $\overline{k}(\delta)$ by
\begin{align*}
\overline{k}(\delta)\leq& \frac{\tfrac{9}{8}n  \epsilon_0^2+3\sqrt{n}\|w_{\epsilon_0}\|(\tau-1) \delta \ln(\delta^{-1}+2)}{c_0(\tau-1)^2 \delta^2}+M+1\\
\leq& \frac{\tfrac{9}{8}n  \epsilon_0^2+6\sqrt{n}\|w_{\epsilon_0}\|(\tau-1) \delta^\frac12 }{c_0(\tau-1)^2 \delta^2}+M+1,\\
\overline{k}(\delta)\geq& \frac{\tfrac{9}{8}n  \epsilon_0^2+3\sqrt{n}\|w_{\epsilon_0}\|(\tau-1) \delta \ln(\delta^{-1}+2)}{c_0(\tau-1)^2 \delta^2}.
\end{align*}
Consequently,
\begin{align*}
\lim_{\delta\to 0^+}\overline{k}(\delta)\delta^2\leq \tfrac{9}{8}c_0^{-1} n (\tau-1)^{-2} \epsilon_0^2 \quad \mbox{and}\quad 
\lim_{\delta\to 0^+}\overline{k}(\delta)^{-1}=0,
\end{align*}
which imply
\begin{align*}
\lim_{\delta\to 0^+}2c_0 n^{-1}\E\big[k(\delta)|\widetilde{Q}\big] \Prob(\widetilde{Q})\delta^2
\leq &2c_0 n^{-1} \big(\tfrac{9}{4}c_0^{-1} n (\tau-1)^{-2} \epsilon_0^2\big)=\tfrac{9}{2}(\tau-1)^{-2}  \epsilon_0^2.
\end{align*}
\underline{Step 3. Vanishing limit on the last term.}
Next, we estimate the last term in \eqref{eqn:cor_e_E} in the bracket. 
Let $$k_{\rm L}(\delta):=\inf_{\theta\in \widetilde{Q}}\{ k(\delta,\theta)\},\quad \forall \delta\in(0,1].$$ 
The choice of $k_{\rm L}(\delta)$ implies that $k_{\rm L}(\delta)$ is independent of the path $\theta$ and $\lim_{\delta\to 0^+}k_{\rm L}(\delta)=\infty$. 
By Lemma \ref{lem:q_k}, we obtain 
\begin{align*}
\|q_{k(\delta)}^\delta\|^2&\leq \|q_{k_{\rm L}(\delta)}^\delta\|^2+\frac{n c_0}{2-c_0\|B\|}\sum_{j=k_{\rm L}(\delta)}^{k(\delta)-1}\|A\Delta_j^\delta\|^2\\
&\leq \|q_{k_{\rm L}(\delta)}^\delta\|^2+\frac{n c_0}{2-c_0\|B\|}\sum_{j=k_{\rm L}(\delta)}^{\infty}\|A\Delta_j^\delta\|^2,
\end{align*}
which yields
\begin{align*}
&\lim_{\delta\to 0^+}\E\big[\|q_{k(\delta)}^\delta\|^2|\widetilde{Q}\big]\Prob(\widetilde{Q}) \\
\leq &\lim_{k_{\rm L}(\delta)\to\infty}\E[\|q_{k_{\rm L}(\delta)}^\delta\|^2|\widetilde{Q}]\Prob(\widetilde{Q})+\frac{n c_0}{2-c_0\|B\|}\lim_{k_{\rm L}(\delta)\to\infty}\sum_{j=k_{\rm L}(\delta)}^{\infty}\E[\|A\Delta_j^\delta\|^2|\widetilde{Q}]\Prob(\widetilde{Q})\\
\leq &\lim_{k_{\rm L}(\delta)\to\infty}\E[\|q_{k_{\rm L}(\delta)}^\delta\|^2]+\frac{n c_0}{2-c_0\|B\|}\lim_{k_{\rm L}(\delta)\to\infty}\sum_{j=k_{\rm L}(\delta)}^{\infty}\E[\|A\Delta_j^\delta\|^2].
\end{align*}
Next, by Lemma \ref{lem:q_k_E} with $k=k_{\rm L}(\delta)$,
\begin{align*}
\lim_{k_{\rm L}(\delta)\to\infty}\E[\|q_{k_{\rm L}(\delta)}^\delta\|^2]\leq& 
c_0(c_1+c_2) (c_0\|B\|+2) M^2\lim_{k_{\rm L}(\delta)\to\infty}  \big(k_{\rm L}(\delta)-1\big)^{-1}=0.
\end{align*}
Further, Lemma \ref{lem:Delta} gives 
\begin{align*}
\lim_{k_{\rm L}(\delta)\to\infty}\sum_{j=k_{\rm L}(\delta)}^{\infty}\E[\|A\Delta_j^\delta\|^2] \leq (c_1+c_2)M^2\lim_{k_{\rm L}(\delta)\to\infty}\sum_{j=k_{\rm L}(\delta)}^{\infty} (j+M)^{-2}=0.
\end{align*}
Finally, by combining the preceding estimates, we derive from \eqref{eqn:cor_e_E} that
\begin{align*}
\lim_{\delta\to 0^+}\E[\|e_{k(\delta)}^\delta\|^2]^\frac12
\leq \sqrt{2}\epsilon_0+3(\tau-1)^{-1}  \epsilon_0\leq \big(\sqrt{2}+3(\tau-1)^{-1}\big) \epsilon_0.
\end{align*}
This completes the proof of the theorem.
\end{proof}

\section{Numerical experiments and discussions}\label{sec:num}
In this section, we provide numerical experiments for several linear inverse problems in Hilbert spaces to complement the theoretical findings in Section \ref{sec:conv}. The experimental setting is identical to that in the works \cite{JinZhouZou:2022ip,JinZhou:2026}. We employ three examples,  i.e., \texttt{s-phillips} (mildly ill-posed), \texttt{s-gravity} (severely ill-posed) and \texttt{s-shaw}
(severely ill-posed), which are generated from the code \texttt{phillips}, \texttt{gravity} and \texttt{shaw}, taken from the \texttt{MATLAB}
package Regutools \cite{P.C.Hansen2007} (publicly available at \url{http://people.compute.dtu.dk/pcha/Regutools/}). 
All the examples are discretized into a finite-dimensional linear system with a forward operator $A: \mathbb{R}^m \rightarrow \mathbb{R}^n$ of size $n = m = 1000$, with $A x= (A_1 x,\cdots,A_n x)$ for all $x\in \mathbb{R}^m$ and $A_i: \mathbb{R}^m \rightarrow \mathbb{R}$.
To precisely control the regularity index $\nu$ in the source condition in Assumption
\ref{ass}(ii), we generate the exact solution $x^\dag$ by 
\begin{equation}\label{eqn:num_gene_x}
x^\dag = \|\big(A^*A\big)^\nu x_e\|_{\ell^\infty}^{-1}\big(A^*A\big)^\nu x_e,
\end{equation}
with $x_e$ being the exact solution provided by the package and $\|\cdot\|_{\ell^\infty}$ denoting the maximum norm of a vector. Note that the index $\nu$ in the source condition is slightly larger than the one used in \eqref{eqn:num_gene_x} due to the existing regularity of $x_e$. The exact data $y^\dag$ is given by $y^\dag=A x^\dag$ and the noisy data $y^\delta$ is generated by
\begin{equation*}
y^\delta_i:=y^\dag_i+\epsilon\|y^\dag\|_{\ell^\infty}\xi_i,\quad i=1,\cdots,n,
\end{equation*}
where the noise components $\xi_i$ follow the standard Gaussian distribution, and $\epsilon > 0$ is the relative noise level. 

All the iterative methods are initialized to zero, with a constant step size $c_0$, depending on the constant $c=(\max_i(\|A_i\|^2))^{-1}$ for SVRG and $c_0=\|A\|^{-2}$ for the Landweber method (LM). 
The constant step size $c_0$ is taken for SVRG to achieve optimal convergence while maintaining computational efficiency across all noise levels. In particular, Assumption \ref{ass}(i) requires the step size $c_0\leq L^{-1}=c$, and smaller step sizes are associated with larger $M$, reflecting the dependence of $\overline{C_0}$ on $M$.
The methods are run for a maximum of 1e5 epochs, where one epoch refers to one Landweber iteration or $nM/(n+M)$ SVRG iterations, so that their overall computational complexity is comparable. 
The frequency $M$ of computing the full gradient is chosen from the set $\{0.1n, n, 2n\}$.
The stopping index $k_*=k(\delta)$ (measured in terms of epoch count) is chosen by the discrepancy principles \eqref{eqn:discrepancy_1} and \eqref{eqn:discrepancy_2}
with $\tau=1.01$ for LM and SVRG, respectively.
Note that the benchmark method, LM, can achieve order optimality \cite{HankeNeubauerScherzer:1995}.
The accuracy of the reconstructions is measured by the square root of the relative mean squared errors $e={\E[\| x_{k}^\delta-x^\dag\|^2]^\frac12}/{\|x^\dag\|}$ at the stopping index $k_*$, i.e., $e_*={\E[\| x_{k_*}^\delta-x^\dag\|^2]^\frac12}/{\|x^\dag\|}$, for SVRG, 
and the relative error $e=\|x_{k}^\delta-x^\dag\|/\|x^\dag\|$ at $k_*$, i.e., $e_*=\|x_{k_*}^\delta-x^\dag\|/\|x^\dag\|$, for LM. 
The statistical quantities generated by SVRG are computed based on ten independent runs. 

The numerical results for the three examples with varying regularity indices $\nu$ and noise levels $\epsilon$, are presented in Tables \ref{tab:phil}, \ref{tab:gravity}, and \ref{tab:shaw}.
It is observed  that, when empolying suitable constant step sizes, SVRG with varying frequency $M$ achieves an accuracy (with much fewer iterations for relatively low-regularity cases) comparable with that for the optimal LM across $\nu\in\{0,0.25,0.5,1\}$. 
Typically, problems with a higher noise level or a higher regularity require fewer iterations. These observations agree with the theoretical results of Theorems \ref{thm:dp_E} and  \ref{thm:dp}, where the (nearly) optimal convergence rates are proven and the upper bounds on the stopping index decrease as the noise level or the regularity index increases. 
However, a smaller step size is required for problems with higher regularity in order to overcome the saturation phenomenon of SVRG \cite{JinZhouZou:2022ip}.
The iteration trajectories of the methods for the three examples with $\nu=0$ in Fig. \ref{fig} show the advantage of SVRG over LM, agreeing well with the quantitative results presented in Tables \ref{tab:phil}-\ref{tab:shaw}. 

\begin{table}[htp!]
  \centering
  \setlength{\tabcolsep}{3pt}
  \begin{threeparttable}
  \caption{The comparison between SVRG and LM for \texttt{s-phillips}.}\label{tab:phil}
    \begin{tabular}{ccccccccccccccc}
    \toprule
    \multicolumn{2}{c}{Method}&
    \multicolumn{3}{c}{SVRG ($M=0.1n$)}&
    \multicolumn{3}{c}{SVRG ($M=n$)}&
    \multicolumn{3}{c}{SVRG ($M=2n$)}&\multicolumn{2}{c}{LM}\\
    \cmidrule(lr){3-5} \cmidrule(lr){6-8} \cmidrule(lr){9-11} \cmidrule(lr){12-13}
    $\nu$& $\epsilon$ &$c_0$ &$e_*$&$k_*$&$c_0$ &$e_*$&$k_*$&$c_0$ &$e_*$&$k_*$&$e_*$&$k_*$\\
    \midrule
    $0$& 1e-3 & c/8 & 1.95e-2 & 244.2  & c/30 & 1.97e-2 & 171.8  & c/40 & 1.97e-2 & 172.8  & 1.93e-2 & 758 \cr
       & 5e-3 &     & 2.99e-2 & 29.1   &      & 3.02e-2 & 19.0  &      & 3.07e-2 & 20.6  & 2.81e-2 & 102 \cr
       & 1e-2 &     & 3.86e-2 & 18.7   &      & 4.32e-2 & 15.3   &      & 3.97e-2 & 14.7   & 3.81e-2 & 68  \cr
       & 5e-2 &     & 9.10e-2 & 4.4    &      & 1.05e-1 & 4.0   &      & 1.00e-1 & 4.4   & 9.44e-2 & 12  \cr
    \midrule
    $0.25$& 1e-3 & c/40 & 4.81e-3 & 193.0  & c/200 & 4.80e-3 & 177.7  & c/300 & 4.84e-3 & 198.5  & 4.58e-3 & 135 \cr
          & 5e-3 &      & 1.50e-2 & 86.3  &       & 1.49e-2 & 76.5   &       & 1.50e-2 & 86.8   & 1.48e-2 & 60  \cr
          & 1e-2 &      & 2.81e-2 & 39.0  &       & 2.79e-2 & 34.1  &       & 2.81e-2 & 40.0  & 2.81e-2 & 26  \cr
          & 5e-2 &      & 4.68e-2 & 14.8  &       & 4.95e-2 & 13.3   &       & 4.92e-2 & 15.0   & 4.66e-2 & 10  \cr
    \midrule
    $0.5$& 1e-3 & c/60 & 2.99e-3 & 201.8  & c/300 & 2.98e-3 & 187.8  & c/500 & 3.01e-3 & 230.1 & 2.90e-3 & 94 \cr
         & 5e-3 &      & 1.23e-2 & 51.7  &       & 1.21e-2 & 46.4  &       & 1.21e-2 & 58.1 & 1.21e-2 & 23 \cr
         & 1e-2 &      & 1.45e-2 & 36.8   &       & 1.53e-2 & 33.8  &       & 1.53e-2 & 42.2  & 1.51e-2 & 16 \cr
         & 5e-2 &      & 2.81e-2 & 18.1   &       & 3.05e-2 & 16.4  &       & 3.01e-2 & 20.4  & 2.92e-2 & 8  \cr
    \midrule
    $1$& 1e-3 & c/120 & 2.03e-3 & 115.5  & c/600 & 1.99e-3 & 104.8  & c/800 & 2.05e-3 & 106.7 & 1.92e-3 & 25 \cr
       & 5e-3 &       & 3.98e-3 & 70.4 &       & 3.74e-3 & 64.0  &       & 3.95e-3 & 64.5  & 3.44e-3 & 16 \cr
       & 1e-2 &       & 5.83e-3 & 55.5   &       & 5.89e-3 & 49.8  &       & 5.88e-3 & 49.8  & 5.54e-3 & 12 \cr
       & 5e-2 &       & 1.68e-2 & 26.4  &       & 1.78e-2 & 23.8  &       & 1.76e-2 & 23.7  & 1.82e-2 & 5  \cr
    \bottomrule
    \end{tabular}
    \end{threeparttable}
\end{table}

\begin{table}[htp!]
  \centering
  \setlength{\tabcolsep}{3pt}
  \begin{threeparttable}
  \caption{The comparison between SVRG and LM for \texttt{s-gravity}.}\label{tab:gravity}
    \begin{tabular}{ccccccccccccccc}
    \toprule
    \multicolumn{2}{c}{Method}&
    \multicolumn{3}{c}{SVRG ($M=0.1n$)}&
    \multicolumn{3}{c}{SVRG ($M=n$)}&
    \multicolumn{3}{c}{SVRG ($M=2n$)}&\multicolumn{2}{c}{LM}\\
    \cmidrule(lr){3-5} \cmidrule(lr){6-8} \cmidrule(lr){9-11} \cmidrule(lr){12-13}
    $\nu$& $\epsilon$ &$c_0$ &$e_*$&$k_*$&$c_0$ &$e_*$&$k_*$&$c_0$ &$e_*$&$k_*$&$e_*$&$k_*$\\
    \midrule
    $0$& 1e-3 & c/10 & 2.44e-2 & 331.6  & c/20 & 2.40e-2 & 137.7  & c/40 & 2.44e-2 & 183.9  & 2.36e-2 & 1649 \cr
       & 5e-3 &      & 3.94e-2 & 113.8  &      & 4.62e-2 & 28.0   &      & 4.34e-2 & 26.4   & 4.04e-2 & 255  \cr
       & 1e-2 &      & 5.27e-2 & 46.2   &      & 5.89e-2 & 16.3   &      & 5.89e-2 & 16.6   & 5.30e-2 & 113  \cr
       & 5e-2 &      & 1.01e-1 & 9.3    &      & 9.97e-2 & 4.1    &      & 1.10e-1 & 4.6    & 9.90e-2 & 22   \cr  
    \midrule
     $0.25$& 1e-3 & c/40 & 6.61e-3 & 275.5  & c/180 & 6.68e-3 & 217.9  & c/200 & 6.81e-3 & 186.0  & 6.50e-3 & 319 \cr
           & 5e-3 &      & 1.72e-2 & 66.5   &       & 1.66e-2 & 46.9   &       & 1.69e-2 & 43.8   & 1.64e-2 & 71  \cr
           & 1e-2 &      & 2.42e-2 & 36.3   &       & 2.38e-2 & 28.6   &       & 2.48e-2 & 24.8   & 2.32e-2 & 43  \cr
           & 5e-2 &      & 5.41e-2 & 9.9    &       & 5.16e-2 & 8.6    &       & 5.42e-2 & 7.0    & 5.35e-2 & 12  \cr
    \midrule
    $0.5$& 1e-3 & c/80 & 3.56e-3 & 212.8  & c/400 & 3.49e-3 & 170.7  & c/500 & 3.64e-3 & 173.7  & 3.39e-3 & 112 \cr
         & 5e-3 &      & 9.85e-3 & 67.1   &       & 9.29e-3 & 60.4   &       & 9.35e-3 & 57.3   & 9.10e-3 & 40  \cr
         & 1e-2 &      & 1.45e-2 & 42.3   &       & 1.35e-2 & 37.0   &       & 1.44e-2 & 34.5   & 1.41e-2 & 25  \cr
         & 5e-2 &      & 3.36e-2 & 14.8   &       & 3.41e-2 & 13.2   &       & 3.33e-2 & 12.2   & 3.40e-2 & 8   \cr
    \midrule
    $1$& 1e-3 & c/160 & 1.60e-3 & 145.7  & c/800 & 1.52e-3 & 128.7  & c/1000 & 1.89e-3 & 125.2  & 1.46e-3 & 42 \cr
       & 5e-3 &       & 4.55e-3 & 59.4   &       & 4.39e-3 & 54.6   &        & 4.54e-3 & 49.8   & 4.11e-3 & 18 \cr
       & 1e-2 &       & 6.74e-3 & 42.3   &       & 6.29e-3 & 37.9   &        & 6.71e-3 & 35.6   & 6.58e-3 & 12 \cr
       & 5e-2 &       & 1.76e-2 & 20.9   &       & 1.78e-2 & 19.1   &        & 1.77e-2 & 17.8   & 1.48e-2 & 6  \cr
 \bottomrule
    \end{tabular}
    \end{threeparttable}
\end{table}

\begin{table}[htp!]
  \centering
  \setlength{\tabcolsep}{3pt}
  \begin{threeparttable}
  \caption{The comparison between SVRG and LM for \texttt{s-shaw}.}\label{tab:shaw}
    \begin{tabular}{ccccccccccccccc}
    \toprule
    \multicolumn{2}{c}{Method}&
    \multicolumn{3}{c}{SVRG ($M=0.1n$)}&
    \multicolumn{3}{c}{SVRG ($M=n$)}&
    \multicolumn{3}{c}{SVRG ($M=2n$)}&\multicolumn{2}{c}{LM}\\
    \cmidrule(lr){3-5} \cmidrule(lr){6-8} \cmidrule(lr){9-11} \cmidrule(lr){12-13}
    $\nu$& $\epsilon$ &$c_0$ &$e_*$&$k_*$&$c_0$ &$e_*$&$k_*$&$c_0$ &$e_*$&$k_*$&$e_*$&$k_*$\\
    \midrule
    $0$& 1e-3 & c & 4.95e-2 & 900.3  & c & 5.01e-2 & 176.5  & c & 5.16e-2 & 119.2  & 4.93e-2 & 22314 \cr
       & 5e-3 &   & 9.19e-2 & 187.0  &   & 9.28e-2 & 51.1   &   & 9.21e-2 & 57.0   & 9.28e-2 & 4858  \cr
       & 1e-2 &   & 1.53e-1 & 37.4   &   & 1.45e-1 & 22.0   &   & 9.53e-2 & 32.1   & 1.53e-1 & 642   \cr
       & 5e-2 &   & 1.76e-1 & 7.1    &   & 1.90e-1 & 12.3  &   & 1.51e-1 & 21.3   & 1.78e-1 & 68    \cr
    \midrule
    $0.25$& 1e-3 & c/6 & 1.65e-2 & 304.1  & c/10 & 1.64e-2 & 124.0  & c/20 & 1.66e-2 & 196.5  & 1.69e-2 & 1218 \cr
          & 5e-3 &     & 2.26e-2 & 36.3   &      & 2.42e-2 & 11.3  &      & 2.37e-2 & 18.1   & 2.24e-2 & 139  \cr
          & 1e-2 &     & 2.44e-2 & 27.5   &      & 2.70e-2 & 9.3   &      & 2.74e-2 & 13.6  & 2.59e-2 & 99   \cr
          & 5e-2 &     & 6.55e-2 & 6.6    &      & 5.36e-2 & 4.1    &      & 6.03e-2 & 5.8    & 7.02e-2 & 24   \cr
    \midrule
    $0.5$& 1e-3 & c/20 & 3.35e-3 & 136.9  & c/100 & 3.25e-3 & 124.8  & c/150 & 3.46e-3 & 139.5  & 3.16e-3 & 169 \cr
         & 5e-3 &      & 8.94e-3 & 66.0   &       & 8.87e-3 & 57.2   &       & 9.04e-3 & 65.1   & 8.83e-3 & 78  \cr
         & 1e-2 &      & 1.69e-2 & 35.7   &       & 1.69e-2 & 28.7   &       & 1.72e-2 & 35.4   & 1.69e-2 & 42  \cr
         & 5e-2 &      & 5.64e-2 & 13.7   &       & 5.63e-2 & 11.5  &       & 5.73e-2 & 13.4  & 5.36e-2 & 16  \cr
    \midrule                               
    $1$& 1e-3 & c/60 & 1.84e-3 & 118.2  & c/200 & 1.91e-3 & 94.7  & c/400 & 1.94e-3 & 136.8  & 1.80e-3 & 54 \cr
       & 5e-3 &      & 6.15e-3 & 63.8   &       & 6.26e-3 & 37.6 &       & 6.42e-3 & 57.9  & 6.13e-3 & 25 \cr
       & 1e-2 &      & 1.18e-2 & 48.4  &       & 1.19e-2 & 29.0  &       & 1.20e-2 & 43.5  & 1.18e-2 & 19 \cr
       & 5e-2 &      & 5.46e-2 & 14.8  &       & 5.40e-2 & 8.7  &       & 5.55e-2 & 12.9   & 5.26e-2 &  6 \cr
 \bottomrule
    \end{tabular}
    \end{threeparttable}
\end{table}

\begin{figure}[hbt!]
\centering
  \setlength{\tabcolsep}{4pt}
\begin{tabular}{ccc}
\includegraphics[width=0.31\textwidth,trim={1.5cm 0 0.5cm 0.5cm}]{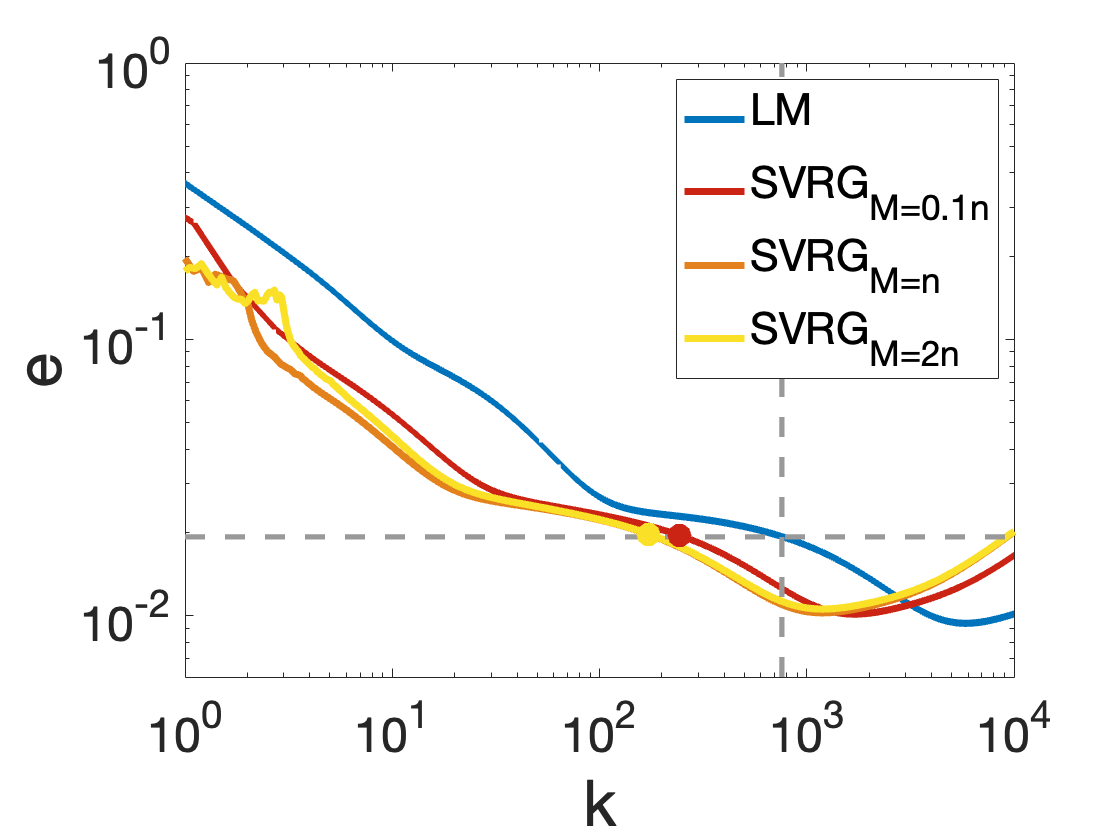}&
\includegraphics[width=0.31\textwidth,trim={1.5cm 0 0.5cm 0.5cm}]{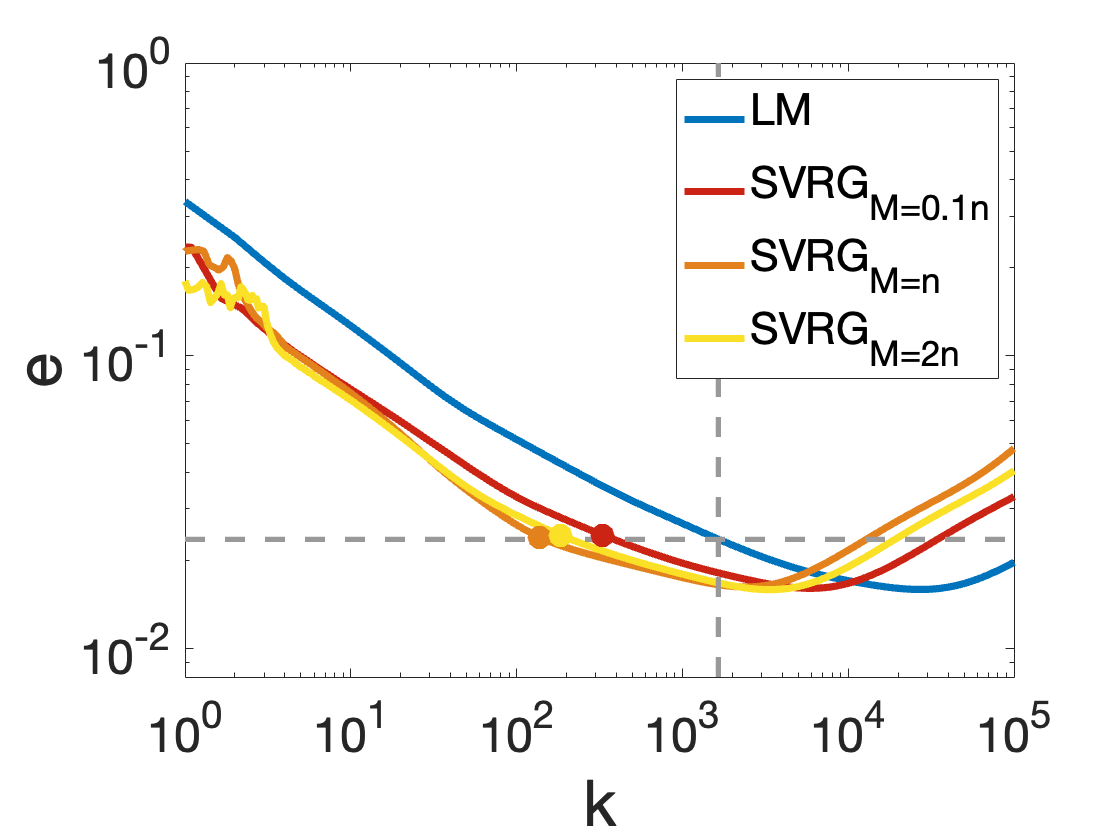}&
\includegraphics[width=0.31\textwidth,trim={1.5cm 0 0.5cm 0.5cm}]{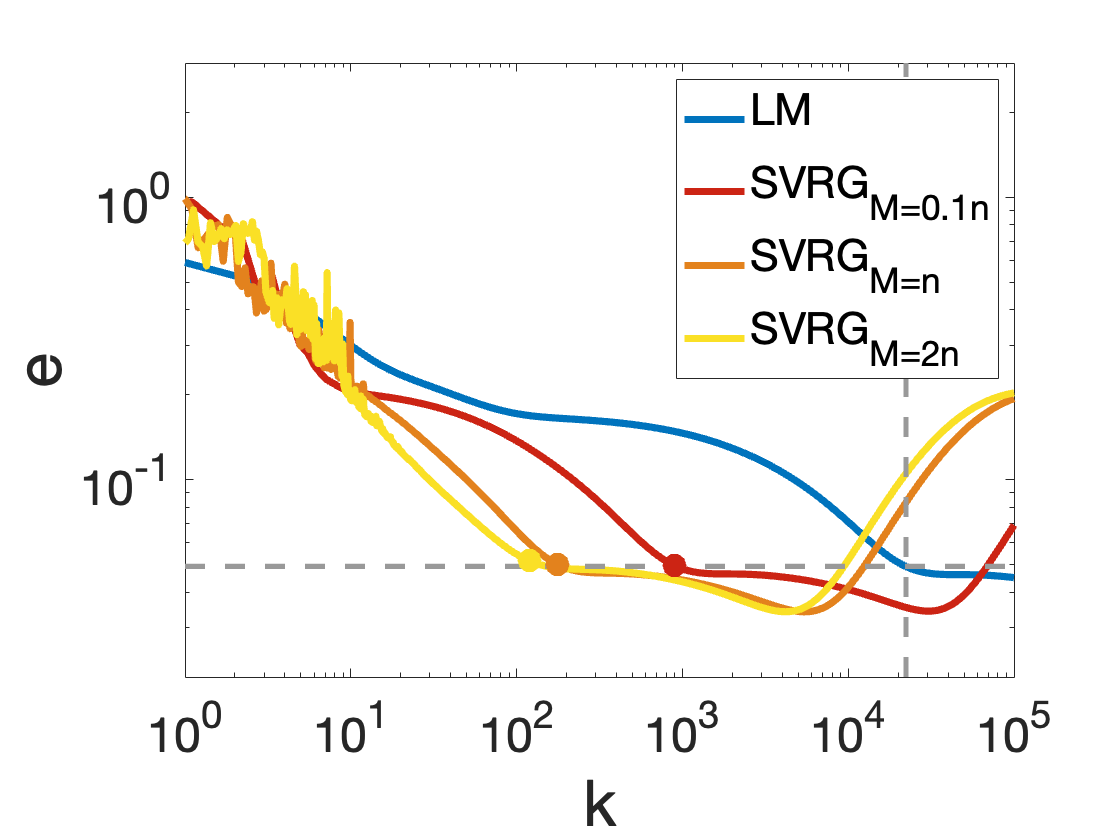}\\
\includegraphics[width=0.31\textwidth,trim={1.5cm 0 0.5cm 0.5cm}]{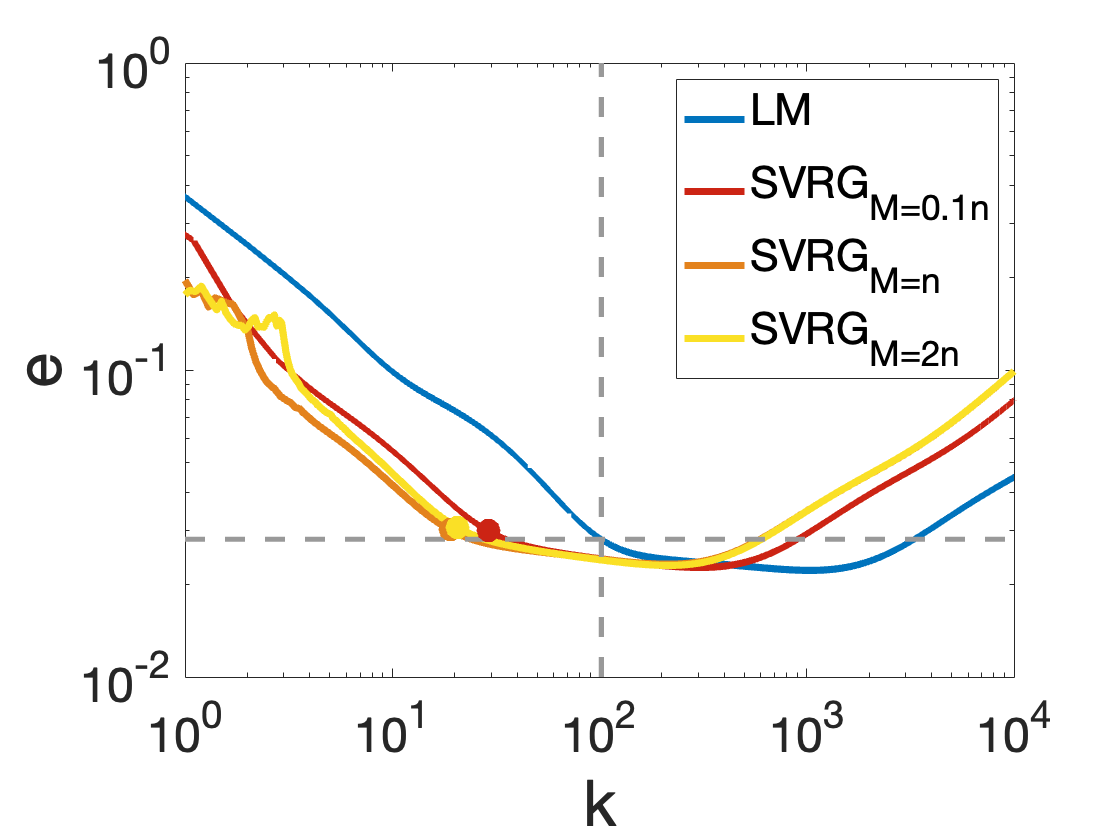}&
\includegraphics[width=0.31\textwidth,trim={1.5cm 0 0.5cm 0.5cm}]{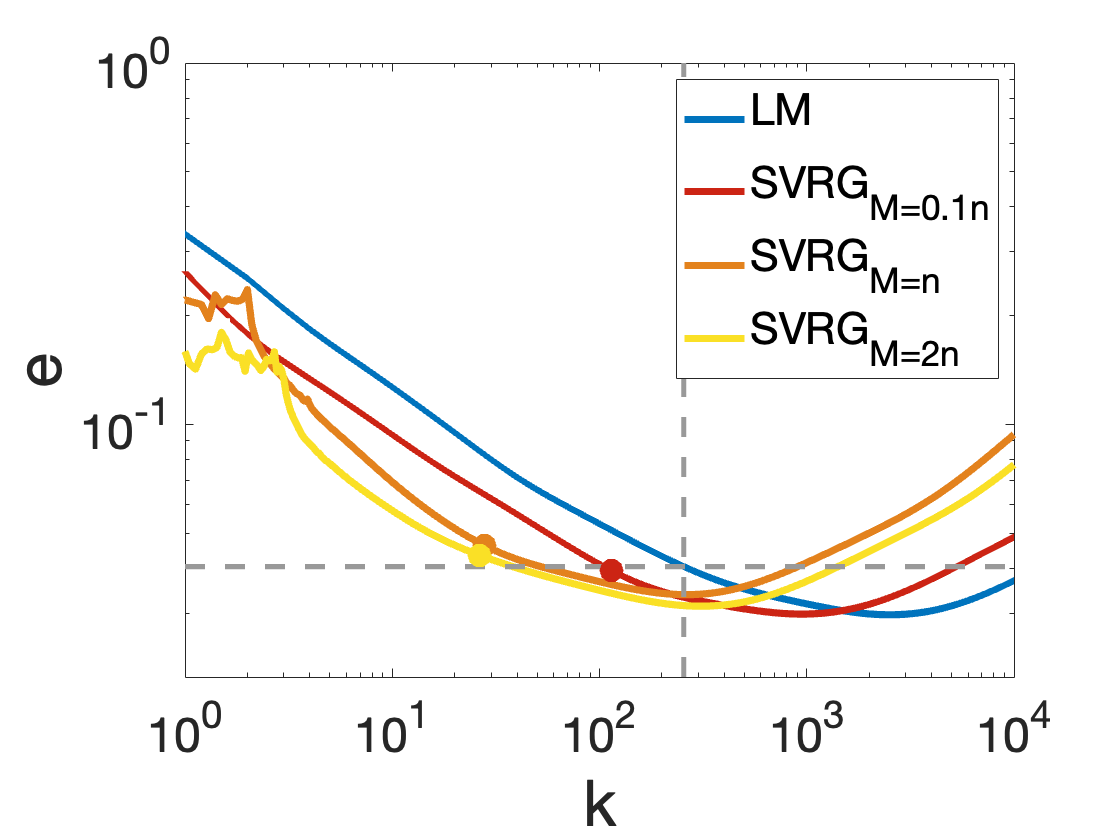}&
\includegraphics[width=0.31\textwidth,trim={1.5cm 0 0.5cm 0.5cm}]{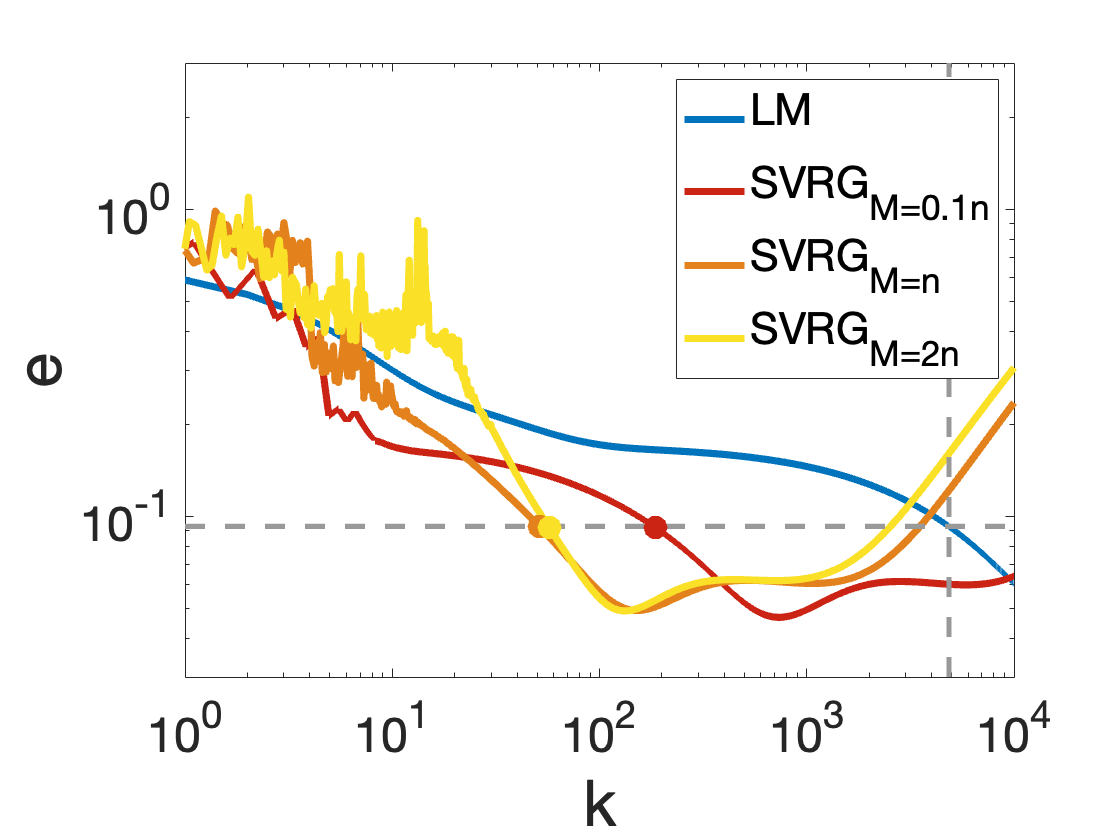}\\
\includegraphics[width=0.31\textwidth,trim={1.5cm 0 0.5cm 0.5cm}]{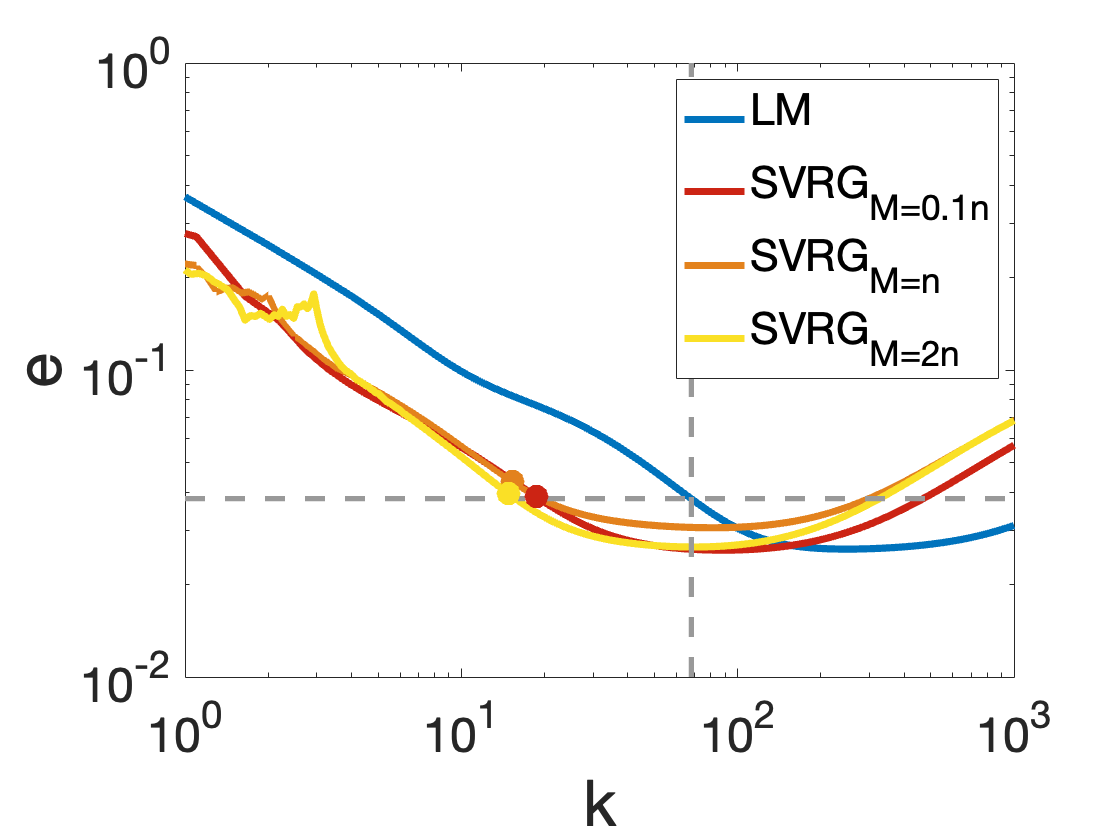}&
\includegraphics[width=0.31\textwidth,trim={1.5cm 0 0.5cm 0.5cm}]{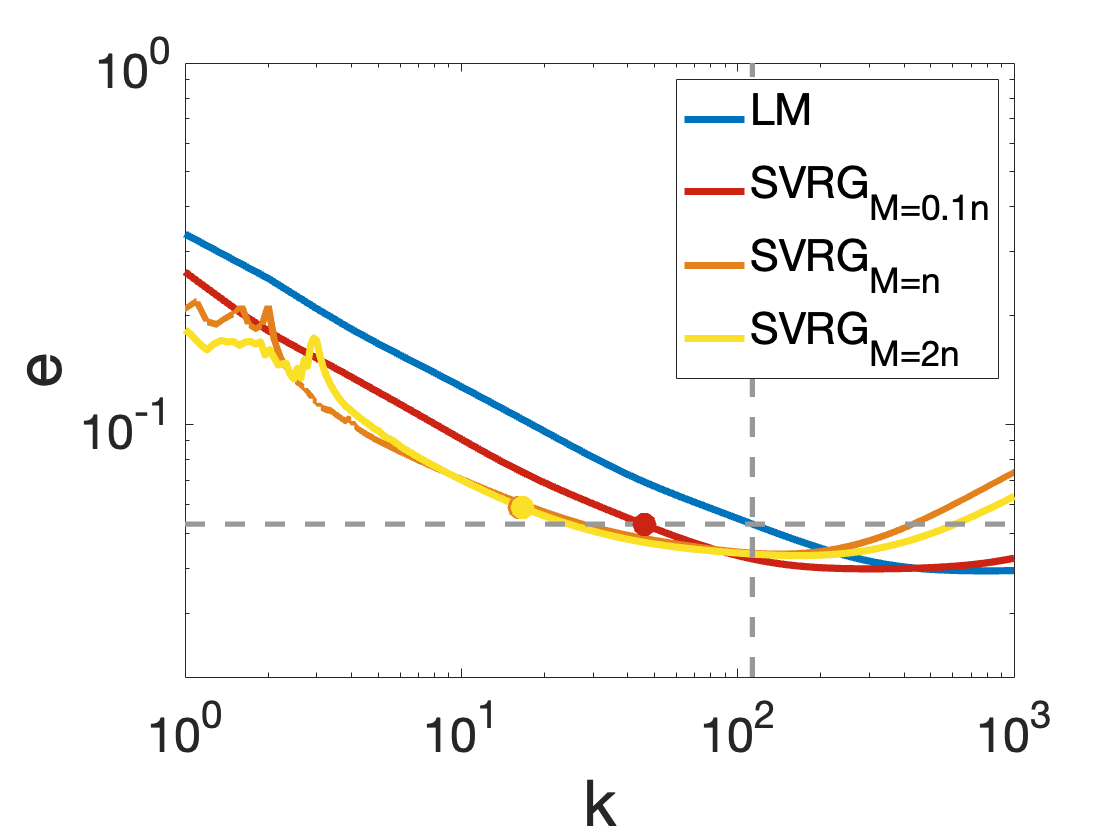}&
\includegraphics[width=0.31\textwidth,trim={1.5cm 0 0.5cm 0.5cm}]{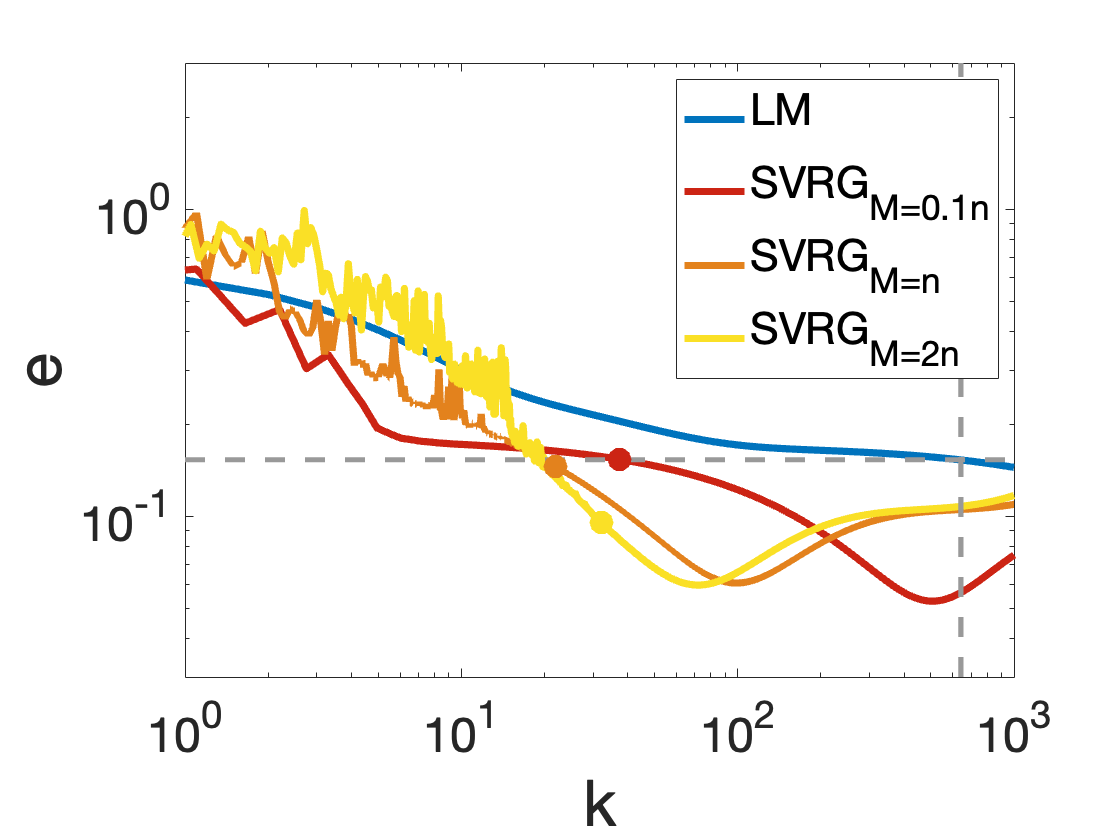}\\
\includegraphics[width=0.31\textwidth,trim={1.5cm 0 0.5cm 0.5cm}]{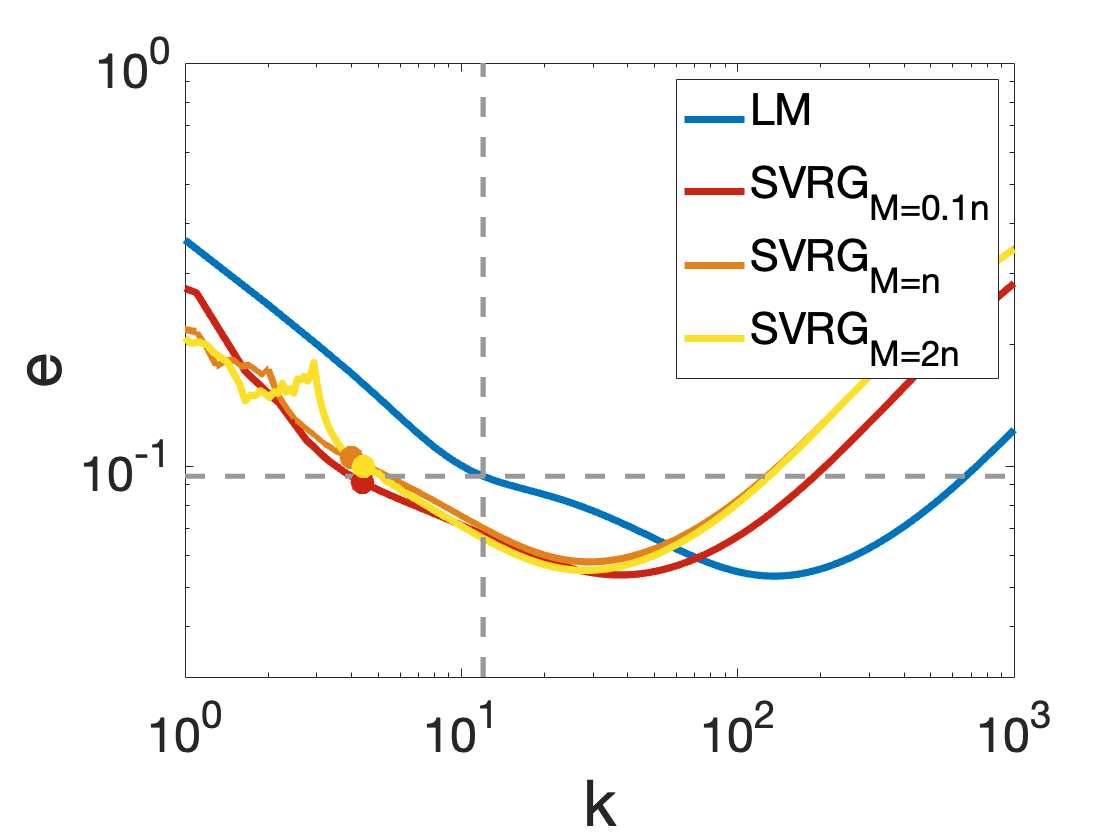}&
\includegraphics[width=0.31\textwidth,trim={1.5cm 0 0.5cm 0.5cm}]{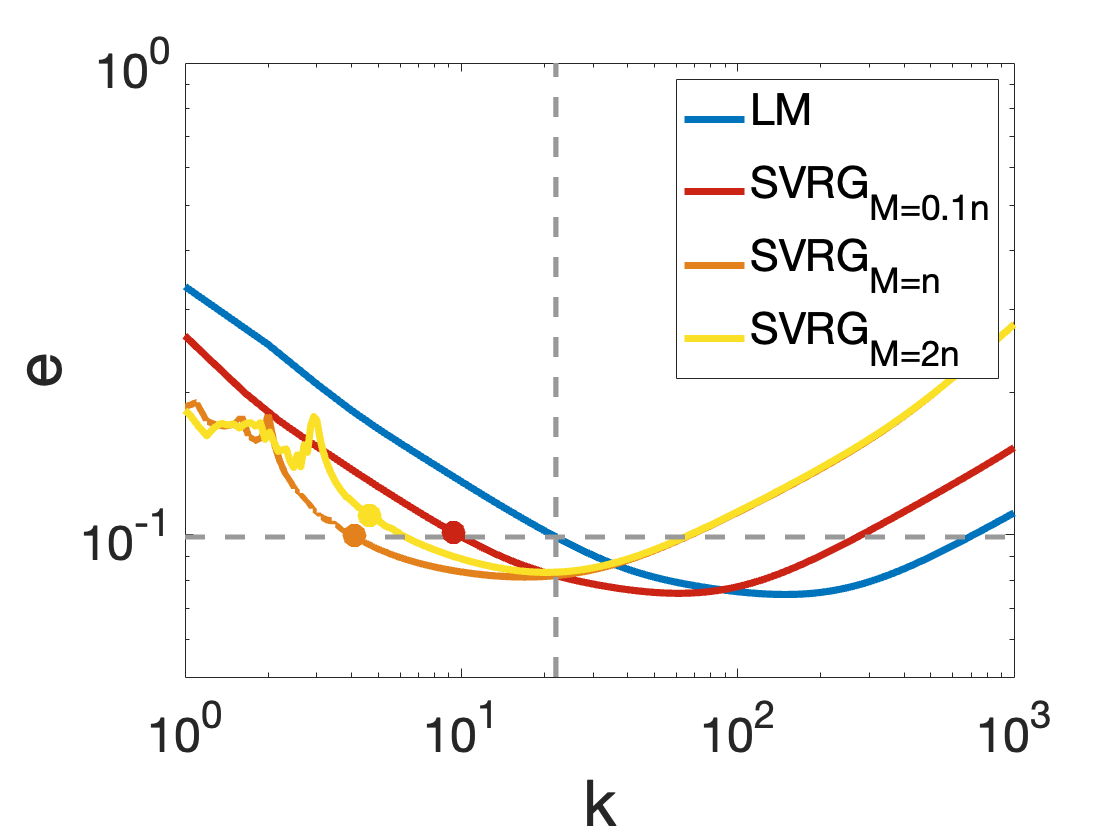}&
\includegraphics[width=0.31\textwidth,trim={1.5cm 0 0.5cm 0.5cm}]{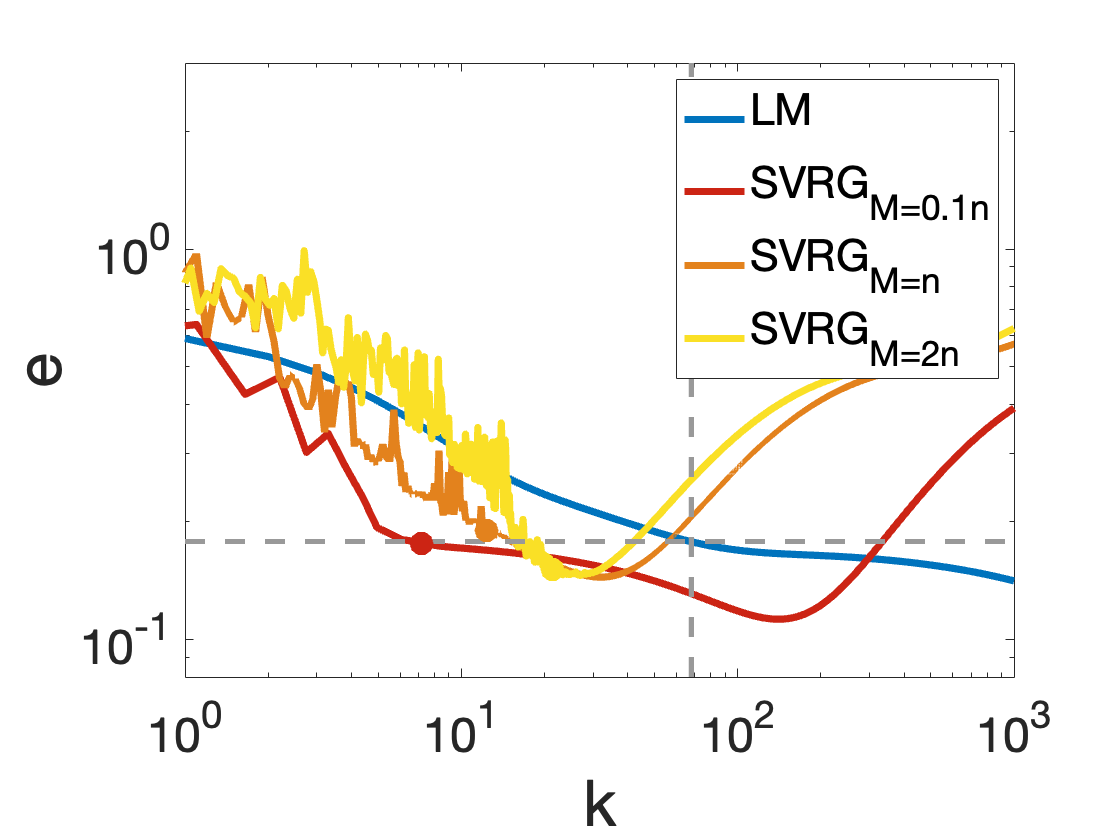}\\
\texttt{phillips}& \texttt{gravity} & \texttt{shaw}
\end{tabular}
\caption{The convergence of the relative error $e={\E[\| x_{k}^\delta-x^\dag\|^2]^\frac12}/{\|x^\dag\|}$ versus the iteration number $k$ for \texttt{phillips}, \texttt{gravity} and \texttt{shaw}. The rows from top to bottom are for $\epsilon=$1e-3,
$\epsilon=$5e-3, 
$\epsilon=$1e-2 and $\epsilon=$5e-2, respectively. The intersection of the gray dashed lines and colored points represent the stopping points, determined by the discrepancy principle, for LM and SVRG along the iteration trajectories, respectively.}\label{fig}
\end{figure}

\section{Concluding remarks} \label{sec:conc}

In this work, we have investigated the convergence of stochastic variance reduced gradient (SVRG) equipped with an \textit{a posteriori} stopping rule, i.e., the discrepancy principle, for solving linear inverse problems in Hilbert spaces. 
Inspired by the technique for analyzing the Landweber method \cite{HankeNeubauerScherzer:1995} and stochastic gradient descent \cite{JahnJin:2020} with the discrepancy principle,  
we have established the finite-iteration termination properties and convergence rates both in probability and in the uniform sense for SVRG with constant step sizes, under the canonical source condition. 
These results indicate the (near) optimality of SVRG equipped with the discrepancy principle for nonsmooth solutions.
Also we have proved the finite-iteration termination and the regularizing properties in the absence of the source condition.
The numerical results for linear inverse problems with varying degree of ill-posedness show the order-optimality and the computational efficiency of SVRG with the discrepancy principle. 

The discrepancy principle requires \textit{a priori} knowledge of the noise level. However, in practice, the noise level may be unknown, making \textit{a posteriori} heuristic stopping rules, e.g., Hanke-Raus rule \cite{HankeRaus:1996} and quasi-optimality criterion \cite{TikhonovGlasko:1964,JinLorenz:2010}, very useful for stochastic iterative methods. We leave this interesting question to future work.

\section*{Acknowledgments} The authors are grateful to two anonymous referees and the associate editor for their many constructive comments on an early version of the paper which have significantly improved the quality of the paper.

\appendix

\section{Preliminary  estimates}\label{app:prelim}

In this section, we provide several preliminary results. 
First we collect several preliminary results from \cite{JinZhou:2026}, including the estimates of the successive error $\Delta_k^\delta$ and the iteration error $e_{k}^\delta$, which are crucial in the analysis.

\begin{lemma}{\cite[Lemma A.1]{JinZhou:2026}}\label{lem:kernel}
Let Assumption \ref{ass}{\rm(i)} hold. Then for any $s> 0$ and $k\in \mathbb{N}$, there holds
\begin{align*}
\|B^s P^k\|\leq s^s c_0^{-s}  k^{-s}.
\end{align*}
\end{lemma}

\begin{lemma}{\cite[Lemma A.2]{JinZhou:2026}}\label{lem:N}
Let $R: X\rightarrow X$ be a deterministic bounded linear operator.  Then for any $k\geq 0$, there hold
\begin{align*}
\E[\|R N_k \Delta_k^\delta\|^2|\mathcal{F}_k]\leq \|R B^\frac12\|^2 \|A\Delta_k^\delta\|^2 \quad\mbox{and}\quad
\|R N_k \Delta_k^\delta\|\leq \sqrt{n}\|R B^\frac12\| \|A\Delta_k^\delta\|.
\end{align*}
\end{lemma}

The next result provides \textit{a priori} bounds on $\E[\|A\Delta_k^\delta\|^2]$ and $\|A\Delta_k^\delta\|$, which are involved in the estimates for both the residual $r_k^\delta$ and the iteration error $e_k^\delta$; see Lemmas \ref{lem:res} and \ref{lem:bias-var}. The result slightly refines \cite[Theorem 3.1]{JinZhou:2026} with a more explicit dependence on the update frequency $M$.
\begin{lemma}{\cite[Theorem 3.1]{JinZhou:2026}}\label{lem:Delta}
Let Assumption \ref{ass}{\rm(i)} hold. Then there exist some  $c_1$ and $c_2$ independent of $k$, $n$, {$M$}, $\delta$ and $\nu$ such that, for any $k\geq 0$, 
\begin{align*}
\E[\|A\Delta_{k}^\delta\|^2]&\leq (c_1+c_2\delta^2){M^2}(k+M)^{-2},  \quad c_0<\overline{C_0}, \\
\|A\Delta_{k}^\delta\|&\leq (c_1+c_2\delta){M}(k+M)^{-1},\quad c_0<C_0. 
\end{align*}
\end{lemma}
\begin{proof}
For the cases $k\leq K_0M$ with some $K_0\geq 1$, the estimates in the lemma hold trivially for sufficiently large $c_1$ and $c_2$. When $k=KM+t$ with $K\geq K_0$ and $t=0,\cdots,M-1$, following the proof of \cite[Theorem 3.1]{JinZhou:2026}, 
we have 
\begin{align*}
\E[\|A\Delta_{k}^\delta\|^2]\leq& \Big[c_0L\max\big(5c_0Mn^{-1}\|A\|^2,\tfrac12+\ln M+2c_0\|B\|+4
(\tfrac72+\tfrac32\ln M)\big)(c_1+c_2\delta^2)M^2\\
&+2(\|A\|^2\|e_0^\delta\|^2+ \delta^2) M^2\Big](1+2K^{-1})^{2}(k+M)^{-2}\\
\leq &(c_1+c_2\delta^2)M^2(k+M)^{-2}, 
\end{align*}
for any $c_0<\overline{C_0}$ and $K\geq1$ with sufficiently large $c_1$ and $c_2$.
Similarly, there holds
\begin{align*}
\|A\Delta_{k}^\delta\|< &\Big[\big(c_0\sqrt{L}M\|A\|+\tfrac{4\sqrt{e}}{7}\big)(c_1+c_2\delta)M+(\|A\|\|e_0^\delta\|+ \delta) M \Big](1+2K^{-1})(k+M)^{-1}\\
\leq &(c_1+c_2\delta)M(k+M)^{-1}
\end{align*}
for any $c_0<C_0$ and $K\geq 35$, with sufficiently large $c_1$ and $c_2$.
This completes the proof of the lemma.
\end{proof}

The following result provides a decomposition of the iteration error $e_k^\delta$, which is useful for bounding the residuals, deriving the convergence rates in Theorems \ref{thm:dp_E} and \ref{thm:dp}, and establishing the regularizing property in Theorem \ref{thm:regularizing}
\begin{lemma}{\cite[Lemma 3.1]{JinZhou:2026}}\label{lem:bias-var}
Let Assumption \ref{ass}{\rm(i)} hold. Then for any $k\geq 0$, there hold
\begin{align*}
\E[e_{k}^\delta]=&P^{k}e_0^\delta+n^{-1}c_0\sum_{j=0}^{k-1}P^j A^* \xi\quad  \mbox{and} \quad
e_{k}^\delta-\E[e_{k}^\delta]=q_k^\delta,
\end{align*}
with the quantity $q_k^\delta$ defined in \eqref{eqn:qk}.
\end{lemma}

The following interpolation inequality is well known \cite[Proposition 8.19]{EnglHankeNeubauer:1996}.
\begin{lemma}\label{lem:moment}
For any $-\infty< p<q<r<\infty$ and $x\in X$, there holds
\begin{equation*}
    \|(A^*A)^qx\| \leq \|(A^*A)^rx\|^{\frac{q-p}{r-p}}\|(A^*A)^px\|^{\frac{r-q}{r-p}}.
\end{equation*}
\end{lemma}

\section{High-probability events}
In this part, we discuss three high-probability events, i.e., the event $\{k(\delta)\leq t\}$ with the stopping index $k(\delta)$ chosen by the discrepancy principle \eqref{eqn:discrepancy_2} for some $t\geq M$, the events $Q$ defined in \eqref{eqn:Q} and $Q_{p_0}$ defined in \eqref{eqn:Q_p}, which are crucial for deriving the convergence rates in expectation in Theorems \ref{thm:dp_E} and \ref{thm:regularizing} (ii).

\begin{lemma}\label{lem:Chebyshev}
Let Assumption \ref{ass}{\rm(i)} hold, $c_0<\overline{C_0}$, and the stopping index $k(\delta)$ be chosen according to the discrepancy principle \eqref{eqn:discrepancy_2} with $\tau>1$. Then for any $0<\tau'<\tau$ and $t\geq M$ such that $\|\E[r^\delta_{M\lfloor M^{-1}t\rfloor}]\|\leq \tau'\delta$, there holds, for any $t'\geq t$, that
\begin{align*}
\Prob(k(\delta)\leq t')\geq& 1-(\tau-\tau')^{-2}n (c^{*})^2\delta^{-2}\lfloor M^{-1}t'\rfloor^{-2}.
\end{align*}
Moreover, there holds
\begin{equation}\label{eqn:Ek}
\E[k(\delta)]\leq 2t+(\tau-\tau')^{-2}n M^2 (c^{*})^2\delta^{-2}t^{-1}.    
\end{equation}
\end{lemma}
\begin{proof}
By the definition of the stopping index $k(\delta)$ in the discrepancy principle \eqref{eqn:discrepancy_2}, the event $\|r^\delta_{t_M}\|\leq \tau \delta$ (with $t_M:= M\lfloor M^{-1}t\rfloor\leq t$) implies 
\begin{align*} 
k(\delta)=M\min\{K\in \mathbb{N}:\;\|r^\delta_{KM}\|\leq \tau \delta\}\leq M\lfloor M^{-1}t\rfloor \leq t.
\end{align*}
Consequently,
\begin{align*}
\Prob\Big(k(\delta)\leq t\Big)\geq \Prob\Big(\|r^\delta_{t_M}\|\leq \tau \delta\Big).
\end{align*}
We decompose the residual $\|r^\delta_{t_M}\|$ by the triangle inequality as $$\|r^\delta_{t_M}\|\leq \|\E[r^\delta_{t_M}]\|+\|r^\delta_{t_M}-\E[r^\delta_{t_M}]\|.$$
Then, under the assumption $\|\E[r^\delta_{t_M}]\|\leq \tau'\delta$, we derive
\begin{align*}
\Prob\Big(\|r^\delta_{t_M}\|\leq \tau \delta\Big)
\geq&
\Prob\Big(\|r^\delta_{t_M}-\E[r^\delta_{t_M}]\|\leq (\tau-\tau') \delta\Big).
\end{align*}
Then combing the preceding estimates yields
\begin{align}\label{eqn:prob}
\Prob\Big(k(\delta)\leq t\Big)
\geq& \Prob\Big(\|r^\delta_{t_M}-\E[r^\delta_{t_M}]\|\leq (\tau-\tau') \delta\Big)
= 1-\Prob\Big(\|r^\delta_{t_M}-\E[r^\delta_{t_M}]\|> (\tau-\tau') \delta\Big).
\end{align}
By Chebyshev's inequality and Theorem \ref{thm:res}(i), we have 
\begin{align*}
&\Prob\Big(\|r^\delta_{t_M}-\E[r^\delta_{t_M}]\|> (\tau-\tau') \delta\Big)
\leq (\tau-\tau')^{-2}\delta^{-2}{\E[\|r^\delta_{t_M}-\E[r^\delta_{t_M}]\|^2]}\\
\leq &(\tau-\tau')^{-2}n M^2(c^{*})^2\delta^{-2}t_M^{-2}
= (\tau-\tau')^{-2}n (c^{*})^2\delta^{-2}\lfloor M^{-1}t\rfloor^{-2},
\end{align*}
which together with the estimate \eqref{eqn:prob} implies
\begin{align}
\Prob\Big(k(\delta)\leq t\Big)
&\geq 1-(\tau-\tau')^{-2}n (c^{*})^2\delta^{-2}\lfloor M^{-1}t\rfloor^{-2}\nonumber\\
&:=1-c_{\tau,\tau',n,M}\delta^{-2}M^{-2}\lfloor M^{-1}t\rfloor^{-2}.\label{eqn:P-kdelta}
\end{align}
Further, by the monotonicity of the expected residual $\|\E[r_{t_M}^\delta]\|$ in Lemma \ref{lem:mono_res}, the inequality \eqref{eqn:P-kdelta} holds also for any $t'\geq t$. Note that the stopping index $k(\delta)$ is measurable with respect to the filtration $\mathcal{F}$, $t\geq M$, and there holds
\begin{align*}
\Prob\Big(k(\delta)> t'\Big)=1-\Prob\Big((k(\delta)\leq t'\Big)
\leq c_{\tau,\tau',n,M}\delta^{-2}M^{-2}\lfloor M^{-1}t'\rfloor^{-2}. 
\end{align*}
Consequently, we derive
\begin{align*}
\E[k(\delta)]
\leq & \big(t+M\big)+\sum_{t'=t+M}^\infty \Prob\Big(k(\delta)> t'\Big)\\
\leq& 2t+c_{\tau,\tau',n,M}\delta^{-2}M^{-2}\sum_{t'=t+M}^\infty \lfloor M^{-1}t'\rfloor^{-2}\\
\leq& 2t+c_{\tau,\tau',n,M}\delta^{-2}\sum_{t'=t+M}^\infty (t'-M+1)^{-2}
\\
\leq& 2t+c_{\tau,\tau',n,M}\delta^{-2}t^{-1}.
\end{align*}
This shows the second estimate in the lemma, and completes the proof of the lemma. 
\end{proof}

\begin{prop}\label{prop:Q}
Let Assumption \ref{ass}{\rm(i)} hold, $c_0<\overline{C_0}$, and the event $Q$ be defined in \eqref{eqn:Q}. Then when $\delta\leq 1$, there hold $$\Prob(Q)\geq \tfrac12 \quad \mbox{and}\quad \lim_{\delta\to 0^+}\Prob(Q)=1.$$
\end{prop}
\begin{proof}
By the definition of the event $Q$, Markov's inequality, and Lemmas \ref{lem:N} and \ref{lem:Delta}, we have
\begin{align*}
\Prob\big(\|B^{-\nu}N_j\Delta_j^\delta\|\geq j^{1-\frac1s} \rho\big)\leq &j^{-2(1-\frac1s)}\rho^{-2}\E[\|B^{-\nu}N_j\Delta_j^\delta\|^2]\\
\leq& j^{-2(1-\frac1s)}\rho^{-2}\|B^{\frac12-\nu}\|^2\E[\|A\Delta_j^\delta\|^2]\\
\leq &\rho^{-2}(c_1+c_2\delta^2)\|A\|^{2-4\nu}n^{-(1-2\nu)}M^2(j+M)^{-2}j^{-2(1-\frac1s)},
\end{align*}
which implies, with $\widehat{k}(\delta):=\lceil\overline{k}(\delta)^{(1+2\nu)s}\rceil$, that 
\begin{align}
\Prob(Q)\geq& 1-
\sum_{j= \hat{k}(\delta)}^\infty \Prob\big(\|B^{-\nu}N_j\Delta_j^\delta\|\geq j^{1-\frac1s} \rho\big)\nonumber\\
\geq& 1-(c_1+c_2\delta^2)\|A\|^{2-4\nu}n^{-(1-2\nu)}\rho^{-2}M^2 \mathfrak{I},\label{eqn:prob-PQ}
\end{align}
with
\begin{align*}
\mathfrak{I}=\overline{k}(\delta)^{-2(1+2\nu)(s-1)}\sum_{j= \overline{k}(\delta)}^\infty (j+M)^{-2}\leq   \overline{k}(\delta)^{-(1+2\nu)(3s-2)}. 
\end{align*}
Then by the definition of the index $\overline{k}(\delta)$ and the condition $\delta\leq 1$, there holds
\begin{align*}
\mathfrak{I}\leq   \bigg(\frac{(\tau-1) \delta}{2\sqrt{n}c_\nu \|w\|}\bigg)^{2(3s-2)}.
\end{align*}
This estimate and \eqref{eqn:prob-PQ} then imply 
\begin{align*}
\Prob(Q)
\geq &1-(c_1+c_2\delta^2)\|A\|^{2-4\nu}n^{-(1-2\nu)}\rho^{-2}M^2 \bigg(\frac{(\tau-1) \delta}{2\sqrt{n}c_\nu \|w\|}\bigg)^{2(3s-2)}\geq 1-\frac12 \delta^{2(3s-2)}.
\end{align*}
Further, since $\delta\leq 1$, we have $\Prob(Q)\geq \frac12$ and $\lim_{\delta\to 0^+}\Prob(Q)=1$.
\end{proof}

\begin{prop}\label{prop:Q_p}
Let Assumption \ref{ass}{\rm(i)} hold, $c_0<\overline{C_0}$, and let the event $Q_{p_0}$ be defined in \eqref{eqn:Q_p} for any $p_0\in(0,1)$. Then when $\delta\leq 1$, there holds $$\Prob(Q_{p_0})\geq 1-p_0.$$
\end{prop}
\begin{proof}
The decomposition in Lemma \ref{lem:bias-var}, the triangle inequality, Corollary \ref{cor:res} and the inequality \eqref{eqn:noise0} yield
\begin{align}
\|e_{k(\delta)}^\delta\|
\leq& \|P^{k(\delta)}e_0^\delta\|+\bigg\|n^{-1}c_0\sum_{j=0}^{k(\delta)-1}P^j A^* \xi\bigg\|+\|q_{k(\delta)}^\delta\|\nonumber\\
\leq& \epsilon_0+\big(2c_0 k(\delta)\big)^{-\frac12}\|w_{\epsilon_0}\|+\sqrt{2c_0 n^{-1} k(\delta)} \delta+\|q_{k(\delta)}^\delta\|.\label{eqn:ek-upper}
\end{align}
There exists some $\delta_0\in(0,1)$ such that for any $\delta<\delta_0$ and any path such that $\lim_{\delta\to 0^+}k(\delta)<\infty$, there holds $k(\delta)\leq \delta^{-1}$, and hence the first three terms in the upper bound \eqref{eqn:ek-upper} of $\|e^\delta_{k(\delta)}\|$ are uniformly bounded when $\delta<\delta_0$.
For any $p_0\in(0,1)$, there exists some $C_m>0$ such that for any $\delta\in(0,1]$ and $k(\delta)\leq p_0^{-1}$, there holds $\|e^\delta_{k(\delta)}\|<C_m$. 
When $k\geq p_0^{-1}+1$, by Chebyshev's inequality and Lemma \ref{lem:q_k_E} below, we deduce
\begin{align*}
\Prob\big(\|q_k^\delta\|\geq \sqrt{c_q} M\big)\leq c_q^{-1} M^{-2}\E[\|q_k^\delta\|^2]
\leq (k-1)^{-1} \leq  p_0,  
\end{align*}
with the constant $c_q=c_0(c_1+c_2) (c_0\|B\|+2)$. 
Thus, when $\lim_{\delta\to 0^+}k(\delta)<\infty$, $\|e^\delta_{k(\delta)}\|$ is not uniformly bounded with probability at most $p_0$, i.e., $\Prob(Q_{p_0})\geq 1-p_0$.
\end{proof}

\section{Technical estimates for the proofs of Theorems \ref{thm:dp_E} and \ref{thm:regularizing}}

The next lemma provides an upper bound of the conditional expectation $\E[\|B^{-\nu}z_{k(\delta)}^\delta\|^2|Q]^\frac12$ with $z_{k(\delta)}^\delta$ defined in \eqref{eqn:decom_dp}, which is crucial for deriving convergence rates in Theorem \ref{thm:dp_E}.
\begin{lemma}\label{lem:z_k}
Let $z_{k(\delta)}^\delta$ be defined in \eqref{eqn:decom_dp}. Then under the conditions in Theorem \ref{thm:dp_E}, there holds
\begin{align*}
\E[\|B^{-\nu}z_{k(\delta)}^\delta\|^2|Q]^\frac12\leq \bar{C}_{w,\nu,\tau} \max\big(1,\tfrac{\tau-1}{2c_\nu \|w\|}\big)^{3s}s n^\nu M\ln\big(\delta^{-1}+M\big),
\end{align*}
where the constant $\bar{C}_{w,\nu,\tau}$ is independent of $s$, $n$, $M$ and $\delta$.
\end{lemma}
\begin{proof}
The decomposition \eqref{eqn:decom_dp}, the triangle inequality, Assumption \ref{ass}(ii) and Lemma \ref{lem:bias-var} together yield
\begin{align}
\E[\|B^{-\nu}z_{k(\delta)}^\delta\|^2|Q]^\frac12
\leq&  \E[\|B^{-\nu}P^{k(\delta)}e_0^\delta\|^2|Q]^\frac12+\E[\|B^{-\nu} q_{k(\delta)}^\delta\|^2|Q]^\frac12\nonumber\\
\leq& \E[\|P^{k(\delta)}w\|^2|Q]^\frac12+\E[\|B^{-\nu} q_{k(\delta)}^\delta\|^2|Q]^\frac12\nonumber\\
\leq &\|w\|+\E[\|B^{-\nu} q_{k(\delta)}^\delta\|^2|Q]^\frac12.\label{eqn:Bz}
\end{align}
Then, by the triangle inequality and the inequalities $\|P^{k(\delta)-1-j}\|\leq 1$ and $\Prob(Q)\geq \frac12$ from Proposition \ref{prop:Q}, we have
\begin{align*}
\E[\|B^{-\nu} q_{k(\delta)}^\delta\|^2|Q]^\frac12
\leq& c_0\E\Big[\Big(\sum_{j=1}^{k(\delta)-1}\|B^{-\nu} P^{k(\delta)-1-j} N_j\Delta_j^\delta\|\Big)^2\Big|Q\Big]^\frac12\\
\leq &c_0\E\Big[\Big(\sum_{j=1}^{k(\delta)-1}\|B^{-\nu}N_j\Delta_j^\delta\|\Big)^2\Big|Q\Big]^\frac12\\
\leq&c_0\Prob(Q)^{-1}\sum_{j=1}^{\widehat{k}(\delta)+M-1}\E[\|B^{-\nu}N_j\Delta_j^\delta\|^2]^\frac12+\mathfrak{I}\\
\leq &2c_0\sum_{j=1}^{2\widehat{k}(\delta)-1}\E[\|B^{-\nu}N_j\Delta_j^\delta\|^2]^\frac12+\mathfrak{I},
\end{align*}
with $\widehat{k}(\delta):=\lceil\overline{k}(\delta)^{(1+2\nu)s}\rceil$ (recall that $\overline{k}(\delta)$ is defined in \eqref{eqn:def_bar_k}) satisfying $\widehat{k}(\delta)+M\leq 2\widehat{k}(\delta)$ and 
\begin{align*}
\mathfrak{I}=&c_0\sum_{j=\widehat{k}(\delta)+M}^{\infty}\E[\chi_{\{k(\delta)> j\}}\|B^{-\nu}N_{j}\Delta_{j}^\delta \|^2|Q]^\frac12,
\end{align*}
with $\chi_{\{k(\delta)> j\}}$ being the indicator function of the set of paths such that $k(\delta)> j$. 
By the definition of the event $Q$ in \eqref{eqn:Q} and Lemma \ref{lem:Chebyshev}, we derive
\begin{align*}
\mathfrak{I}
\leq&c_0\sum_{j=\widehat{k}(\delta)+M}^\infty \Prob\big(k(\delta)> j\;|\; Q\big) j^{1-\frac1s}\rho\\
\leq& c_0\rho\Prob(Q)^{-1}\sum_{j=\widehat{k}(\delta)+M}^\infty \Prob\big(k(\delta)> j\big) j^{1-\frac1s}\\
\leq&2c_0\rho c_{\tau,n,M}\delta^{-2}M^{-2}\sum_{j=\widehat{k}(\delta)+M}^\infty j^{-(\frac1s-1)}\lfloor M^{-1}j\rfloor^{-2},
\end{align*}
with the constant $c_{\tau,n,M}=4(\tau-1)^{-2}n M^2(c^{*})^2$.
Direct computation gives
\begin{align*}
    M^{-2}\sum_{j=\widehat{k}(\delta)+M}^\infty j^{-(\frac1s-1)}\lfloor M^{-1}j\rfloor^{-2}\leq \sum_{j=\widehat{k}(\delta)+M}^\infty (j-M+1)^{-(\frac1s+1)}\le s\widehat{k}(\delta)^{-\frac1s} \leq s\overline{k}(\delta)^{-(1+2\nu)}.
\end{align*}\
Consequently, we arrive at
\begin{align*}
    \mathfrak{I} \leq 2c_0\rho c_{\tau,n,M}s\delta^{-2}\overline{k}(\delta)^{-(1+2\nu)}.
\end{align*}
Further, by Lemma \ref{lem:N}, there holds
\begin{align*}
\E[\|B^{-\nu} q_{k(\delta)}^\delta\|^2|Q]^\frac12
\leq&2c_0 \Big(\|B^{\frac12-\nu}\|\sum_{j=1}^{2\widehat{k}(\delta)-1}\E[\|A\Delta_j^\delta\|^2]^\frac12+\rho c_{\tau,n,M}s\delta^{-2}\overline{k}(\delta)^{-(1+2\nu)}\Big).
\end{align*}
Hence, from Lemma \ref{lem:Delta} (when $c_0<\overline{C_0}$) and the trivial inequality $(j+M)^{-1}\leq(j+1)^{-1}$, we deduce
\begin{align}
&\E[\|B^{-\nu} q_{k(\delta)}^\delta\|^2|Q]^\frac12\nonumber\\
\leq& 2c_0 \bigg(\sqrt{c_1+c_2}\|A\|^{1-2\nu}n^{-(\frac12-\nu)}M\sum_{j=1}^{2\widehat{k}(\delta)-1}(j+1)^{-1}+ \rho c_{\tau,n,M}s\delta^{-2}\overline{k}(\delta)^{-(1+2\nu)}\bigg)\nonumber\\
\leq& 2c_0 \Big(2\sqrt{c_1+c_2}\|A\|^{1-2\nu}n^{-(\frac12-\nu)} M (1+2\nu)s\ln \big(\overline{k}(\delta)\big)+ \rho c_{\tau,n,M} s\delta^{-2}\overline{k}(\delta)^{-(1+2\nu)}\Big).\label{eqn:I0}
\end{align}
Next we bound the factors $\ln \big(\overline{k}(\delta)\big)$ and $\delta^{-2}\overline{k}(\delta)^{-(1+2\nu)}$ separately. 
Let $$c_{w,\nu,\tau}:=\max\bigg(e,\Big(\frac{2c_\nu\|w\|}{\tau-1}\Big)^{\frac{2}{1+2\nu}}\bigg).$$ 
Then with $\overline{k}(\delta)$ defined in \eqref{eqn:def_bar_k} and the assumption $\delta\leq1$, there hold
\begin{align}
\overline{k}(\delta)\leq&  \bigg(\frac{2\sqrt{n}c_\nu\|w\|}{\tau-1}\bigg)^{\frac{2}{1+2\nu}}\delta^{-1}\big(\delta^{-1}\big)^{\max(\frac{2}{1+2\nu}-1,0)}+M+1\nonumber\\
\leq& c_{w,\nu,\tau} n^{\frac{1}{1+2\nu}}\delta^{-1}\big(\delta^{-1}\big)^{\max(\frac{2}{1+2\nu}-1,0)}+M+1\nonumber\\
\leq& c_{w,\nu,\tau} \Big(n^{\frac{1}{1+2\nu}}\delta^{-1}\big(\delta^{-1}\big)^{\max(\frac{2}{1+2\nu}-1,0)}+ M\Big)\label{eqn:bar_k_bd}\\
\leq& c_{w,\nu,\tau} n^{\frac{1}{1+2\nu}}\big(\delta^{-1}+M\big)\big(\delta^{-1}\big)^{\max(\frac{2}{1+2\nu}-1,0)}\nonumber\\
\leq& c_{w,\nu,\tau} n^{\frac{1}{1+2\nu}}\big(\delta^{-1}+M\big)^2.\nonumber
\end{align}
Then, by the inequality $\ln (ab)\leq 2\ln a\ln b$ for any $a,b\geq e$, we derive
\begin{align}
\ln \big(\overline{k}(\delta)\big)
\leq& \ln\Big(c_{w,\nu,\tau} n^{\frac{1}{1+2\nu}}\big(\delta^{-1}+M\big)^2\Big)
\leq 2\ln\big(c_{w,\nu,\tau} n^{\frac{1}{1+2\nu}}\big)\ln\big(\delta^{-1}+M\big)^2\nonumber\\
\leq& 4\ln\big(c_{w,\nu,\tau} e n^{\frac{1}{1+2\nu}}\big)\ln\big(\delta^{-1}+M\big)
\leq 8\ln c_{w,\nu,\tau} \ln \big(e n^{\frac{1}{1+2\nu}}\big)\ln\big(\delta^{-1}+M\big).\label{eqn:ln_bark}
\end{align}
Meanwhile, the condition $\overline{k}(\delta)\geq \big(\frac{2\sqrt{n}c_\nu \|w\|}{(\tau-1) \delta}\big)^{\frac{2}{1+2\nu}}$ implies
\begin{align}\label{eqn:bark-2}
\delta^{-2}\overline{k}(\delta)^{-(1+2\nu)}\leq \delta^{-2}\bigg(\frac{2\sqrt{n}c_\nu \|w\|}{(\tau-1) \delta}\bigg)^{-2}
= \bigg(\frac{\tau-1}{2\sqrt{n}c_\nu \|w\|}\bigg)^2.
\end{align}
Consequently, with  the estimates \eqref{eqn:Bz}, \eqref{eqn:I0}, \eqref{eqn:ln_bark}, the definitions $c_{\tau,n,M}=4(\tau-1)^{-2}n M^2(c^{*})^2$ and $\rho=\sqrt{2(c_1+c_2\delta^2)}\|A\|^{1-2\nu}\left(\frac{\tau-1}{2c_\nu \|w\|}\right)^{3s-2}n^{-\frac{3s-1-2\nu}{2}}M$, and the assumptions $\tau\in(1,2]$ and $s\geq 1$, there holds
\begin{align}
&{\E[\|B^{-\nu}z_{k(\delta)}^\delta\|^2|Q]^\frac12}\nonumber \\
{\leq}& {2^5 c_0\sqrt{c_1+c_2}\ln c_{w,\nu,\tau}(1+2\nu)\|A\|^{1-2\nu}s n^{-(\frac12-\nu)}M\ln \big(e n^{\frac{1}{1+2\nu}}\big)\ln\big(\delta^{-1}+M\big)}\nonumber\\
&{+2^{\frac72}c_0 \sqrt{c_1+c_2}(c^{*})^2 (\tau-1)^{-2}\big(\tfrac{\tau-1}{2c_\nu \|w\|}\big)^{3s}\|A\|^{1-2\nu}s n^{-\frac{3s-1-2\nu}{2}}M^3+\|w\|}\nonumber \\
{\leq}&{\overline{C}_{w,\nu,\tau}s\max\Big(1,\tfrac12 n^{-(\frac12-\nu)}\ln \big(e n^{\frac{1}{1+2\nu}}\big), \big(\tfrac{\tau-1}{2c_\nu \|w\|}\big)^{3s}n^{-\frac{3s-1-2\nu}{2}}M^2\Big)M\ln\big(\delta^{-1}+M\big)},\label{eqn:Bz_k}
\end{align}
where the constant  $\overline{C}_{w,\nu,\tau}$ is independent of $s$, $n$, $M$ and $\delta$. 
Further, by the elementary inequality 
\begin{equation}\label{eqn:ln}
\hat{r}^{-r}\ln \hat{r}\leq (er)^{-1}, \quad \forall \hat{r},r > 0,
\end{equation}
we can bound the constant $\tfrac12 n^{-(\frac12-\nu)}\ln \big(e n^{\frac{1}{1+2\nu}}\big)$ in \eqref{eqn:Bz_k} by \begin{align*}
\tfrac12 n^{-(\frac12-\nu)}\ln \big(e n^{\frac{1}{1+2\nu}}\big)=&
\tfrac12 n^\nu e^{\frac{1+2\nu}{2}}(e n^{\frac{1}{1+2\nu}})^{-\frac{1+2\nu}{2}}\ln \big(e n^{\frac{1}{1+2\nu}}\big)
\leq \tfrac12 n^\nu e^{\frac{1+2\nu}{2}}(e\tfrac{1+2\nu}{2})^{-1}\leq n^\nu.   
\end{align*}
This estimate and the assumption $s\geq \max\big(1,\frac13(1+4\ln M/\ln n)\big)$ imply 
\begin{align*}
&\max\Big(1,\tfrac12 n^{-(\frac12-\nu)}\ln \big(e n^{\frac{1}{1+2\nu}}\big),\big(\tfrac{\tau-1}{2c_\nu \|w\|}\big)^{3s}n^{-\frac{3s-1-2\nu}{2}}M^2\Big)\\
\leq& n^\nu\max\big(1,\big(\tfrac{\tau-1}{2c_\nu \|w\|}\big)^{3s}n^{-\frac{3s-1}{2}}M^2\big)
\leq n^\nu\max\big(1,\tfrac{\tau-1}{2c_\nu \|w\|}\big)^{3s}.    
\end{align*}
This and \eqref{eqn:Bz_k} complete the proof of the lemma.
\end{proof}

The next lemma collects some basic estimates, which are essential for the proof of Theorems \ref{thm:dp},  \ref{thm:regularizing} and \ref{thm:res}. Their derivations involve routine but rather lengthy and tedious computations. 

\begin{lemma}\label{lem:sums}
For any $k\geq 1$, with the convention $\sum_{j=1}^{0} R_j=0$ for any sequence $\{R_j\}_j$, there hold
\begin{align}
{(k+M)^{-1}+\sum_{j=1}^{k-1}(k-j)^{-1}(j+M)^{-1}\leq}& {6 (k+1)^{-1}\ln (k+1)},\label{eqn:sum_1_1}\\
(k+M)^{-2}+\sum_{j=1}^{k-1}(k-j)^{-2}(j+M)^{-2}\leq& 8(k+1)^{-2}\label{eqn:sum_2_2},
\end{align}
and for any $k\geq 3$, there hold
\begin{align}
{\sum_{j=1}^{k-1}(k-j)^{-(\frac12-\nu)} (j+M)^{-1}
\leq} & {9(1-2\nu)^{-1}k^{-\frac{1-2\nu}{4}}},\label{eqn:sum_nu_1}\\
\sum_{j=1}^{k-1}  \big(k-j\big)^{-1}(j+M)^{-2}\leq 4 k^{-1}.\label{eqn:sum_1_2}
\end{align}
\end{lemma}

\begin{proof}
Using the identity $(j'-j)^{-1}j^{-1}=(j')^{-1}\big((j'-j)^{-1}+j^{-1}\big)$ (for any $j'>j$), for any $k\geq 1$,  we have
\begin{align*}
& (k+M)^{-1}+\sum_{j=1}^{k-1}(k-j)^{-1}(j+M)^{-1}\\
\leq & (k+M)^{-1}\bigg(1+\sum_{j=M+1}^{k+M-1}[(k+M-j)^{-1}+j^{-1}]\bigg)\\
\leq & (k+1)^{-1}\bigg(1+\sum_{j=1}^{k-1}j^{-1}+\sum_{j=M+1}^{k+M-1}j^{-1}\bigg)
\leq (k+1)^{-1}\bigg(\sum_{j=1}^{k-1}j^{-1}+\sum_{j=1}^{k}j^{-1}\bigg)\\
\leq& 2(k+1)^{-1}\big(1+\ln (k+1)\big)
\leq 2\big((\ln 2)^{-1}+1\big) (k+1)^{-1}\ln (k+1)
\leq 6 (k+1)^{-1}\ln (k+1).
\end{align*}
Next,
using the elementary inequality 
$(j'-j)^{-2}j^{-2}\leq 2(j')^{-2}\big((j'-j)^{-2}+j^{-2}\big),$ we deduce
\begin{align*}
&(k+M)^{-2}+\sum_{j=1}^{k-1}(k-j)^{-2}(j+M)^{-2}\\
=&(k+M)^{-2}+\sum_{j=M+1}^{k+M-1}(k+M-j)^{-2}j^{-2}\\
\leq& 2(k+M)^{-2}\bigg(\frac12+\sum_{j=M+1}^{k+M-1}\big((k+M-j)^{-2}+j^{-2}\big)\bigg)\\
\leq &2(k+1)^{-2}\bigg(\frac12+\sum_{j=1}^{k-1}j^{-2}+\sum_{j=M+1}^{k+M-1}j^{-2}\bigg)\\
\leq &2\big(\tfrac12+\tfrac{\pi^2}{6}+M^{-1}\big)(k+1)^{-2}
\leq  8(k+1)^{-2},
\end{align*}
which gives the estimate \eqref{eqn:sum_2_2}.
Finally, for any $k\geq 3$, direct computation gives
\begin{align*}
&\sum_{j=1}^{k-1}(k-j)^{-(\frac12-\nu)}(j+M)^{-1}\\
=&\sum_{j=1}^{\lfloor\frac{k-1}{2}\rfloor}(k-j)^{-(\frac12-\nu)} (j+M)^{-1}+\sum_{j=\lfloor\frac{k-1}{2}\rfloor+1}^{k-1}(k-j)^{-(\frac12-\nu)} (j+M)^{-1}\\
\leq& \Big(\frac{k}{2}\Big)^{-(\frac12-\nu)}\sum_{j=1}^{\lfloor\frac{k-1}{2}\rfloor} j^{-1}+\Big(\frac{k}{2}+M\Big)^{-1}\sum_{j=\lfloor\frac{k-1}{2}\rfloor+1}^{k-1}(k-j)^{-(\frac12-\nu)}\\
\leq& \Big(\frac{k}{2}\Big)^{-(\frac12-\nu)}\ln\bigg(\frac{e k}{2}\bigg)+2\Big(\frac{k}{2}+M\Big)^{-1}\bigg(\frac{k}{2}\bigg)^{\frac12+\nu}\\
\leq& 2\Big(\frac{k}{2}\Big)^{-(\frac12-\nu)}\big(1+\ln k\big)
\leq  2^{\frac52-\nu}k^{-(\frac12-\nu)}\ln k.
\end{align*} 
Further, the inequality \eqref{eqn:ln} with $r=\frac{1-2\nu}{4}$ yields
\begin{align*}
&\sum_{j=1}^{k-1}(k-j)^{-(\frac12-\nu)} (j+M)^{-1}
\leq 2^{\frac52}k^{-\frac{1-2\nu}{4}}(k^{-\frac{1-2\nu}{4}}\ln k)\\
\leq& 2^{\frac92}k^{-\frac{1-2\nu}{4}}
e^{-1}(1-2\nu)^{-1}
\leq 9(1-2\nu)^{-1}k^{-\frac{1-2\nu}{4}},
\end{align*} 
which implies the estimate \eqref{eqn:sum_nu_1}.
Similarly, we derive the estimate \eqref{eqn:sum_1_2}:
\begin{align*}
\sum_{j=1}^{k-1}  \big(k-j\big)^{-1}(j+M)^{-2}=&\sum_{j=1}^{\lfloor \frac{k-1}{2}\rfloor}  \big(k-j\big)^{-1}(j+M)^{-2}+\sum_{j=\lfloor \frac{k-1}{2}\rfloor+1}^{k-1}  \big(k-j\big)^{-1}(j+M)^{-2}\\
\leq &\Big(\frac{k}{2}\Big)^{-1}\sum_{j=1}^{\lfloor \frac{k-1}{2}\rfloor}  (j+M)^{-2}+(\frac{k}{2}+M)^{-2}\sum_{j=\lfloor \frac{k-1}{2}\rfloor+1}^{k-1}  \big(k-j\big)^{-1}\\
\leq &2 k^{-1}+4k^{-2}\ln\bigg(\frac{e k}{2}\bigg)= 2 k^{-1}+2ek^{-1}\Big(\frac{e k}{2}\Big)^{-1}\ln\bigg(\frac{e k}{2}\bigg)
\leq 4 k^{-1}.
\end{align*}
This completes the proof of the lemma.
\end{proof}

The next two lemmas provide bounds on $q_k^\delta$, which are used in the proof of Theorem \ref{thm:regularizing}.
First we derive a useful relation for $\{\|q_k^\delta\|\}_{k\geq 0}$.
\begin{lemma}\label{lem:q_k}
Let Assumption \ref{ass}{\rm(i)} hold. Then for any $k,k'\in \mathbb{N}$ and $k'\geq k$, there holds
\begin{equation*}
\|q_{k'}^\delta\|^2\leq \|q_k^\delta\|^2+\frac{n c_0}{2-c_0\|B\|}\sum_{j=k}^{k'-1}\|A\Delta_j^\delta\|^2.
\end{equation*}
\end{lemma}
\begin{proof}
By the definitions of $q_j^\delta$ and $P$ in \eqref{eqn:qk}, we derive
\begin{align*}
q_{k+1}^\delta= (I-c_0B) q_k^\delta+c_0 N_k\Delta_k^\delta.
\end{align*}
The identity $\|q_{k+1}^\delta\|^2-\|q_k^\delta\|^2=2\langle q_{k+1}^\delta-q_k^\delta, q_k^\delta \rangle+\| q_{k+1}^\delta-q_k^\delta\|^2$ yields
\begin{align*}
\|q_{k+1}^\delta\|^2-\|q_k^\delta\|^2=-2c_0\langle B q_k^\delta- N_k\Delta_k^\delta, q_k^\delta \rangle+c_0^2\|B q_k^\delta- N_k\Delta_k^\delta\|^2.
\end{align*}
Lemma \ref{lem:N} implies $\|B^{-\frac12}N_k\Delta_k^\delta\|\leq \sqrt{n}\|A\Delta_k^\delta\|$. Hence 
\begin{align*}
&\|q_{k+1}^\delta\|^2\leq \|q_k^\delta\|^2-2c_0\langle B^\frac12 q_k^\delta- B^{-\frac12}N_k\Delta_k^\delta, B^\frac12 q_k^\delta \rangle+c_0^2\|B\|\|B^\frac12 q_k^\delta- B^{-\frac12}N_k\Delta_k^\delta\|^2\\
\leq& \|q_k^\delta\|^2-(2-c_0\|B\|)c_0\| B^\frac12 q_k^\delta- B^{-\frac12}N_k\Delta_k^\delta\|^2-2c_0\langle B^\frac12 q_k^\delta- B^{-\frac12}N_k\Delta_k^\delta, B^{-\frac12}N_k\Delta_k^\delta \rangle\\
\leq& \|q_k^\delta\|^2-(2-c_0\|B\|)c_0\| B^\frac12 q_k^\delta- B^{-\frac12}N_k\Delta_k^\delta\|^2+2\sqrt{n}c_0\|B^\frac12 q_k^\delta- B^{-\frac12}N_k\Delta_k^\delta\| \|A\Delta_k^\delta\|.
\end{align*}
Further, Young's inequality $2ab\leq c a^2 + c^{-1} b^2$ (for any $c>0$),
with the choice $a=\sqrt{c_0}\|B^\frac12 q_k^\delta- B^{-\frac12}N_k\Delta_k^\delta\|$, $b=\sqrt{n c_0}\|A\Delta_k^\delta\|$ and $c= 2-c_0\|B\|>0$ gives 
\begin{align*}
\|q_{k+1}^\delta\|^2
\leq& \|q_k^\delta\|^2+\frac{n c_0}{2-c_0\|B\|}\|A\Delta_k^\delta\|^2
.
\end{align*}
Thus for any $k'\geq k$,  taking the telescopic sum from $k$ to $k'$ yields the desired assertion.
\end{proof}

The following lemma gives the upper bounds of $\|q_{k}^\delta\|$ and $\E[\|q_{k}^\delta\|^2]$.

\begin{lemma}\label{lem:q_k_E}
Let Assumption \ref{ass}{\rm(i)} hold. Then for any $k\geq 3$ and $\delta\leq 1$, there hold
\begin{align*}
\|q_{k}^\delta\|\leq &(c_1+c_2)\big(c_0\|B\|^\frac12 +7 c_0^\frac12\big)\sqrt{n}M(k-1)^{-\frac{1}{4}}, \quad c_0<C_0,\\
\E[\|q_{k}^\delta\|^2] \leq & c_0(c_1+c_2) (c_0\|B\|+2) M^2 \big(k-1\big)^{-1}, \quad c_0<\overline{C_0}.
\end{align*}
\end{lemma}
\begin{proof}
For any $k\geq 3$, by the definition of $q_k^\delta$, the triangle inequality, Lemmas \ref{lem:kernel}, \ref{lem:N} and \ref{lem:Delta} (when $c_0<C_0$), we can bound the term $\|q_{k}^\delta\|$ by
\begin{align*}
\|q_{k}^\delta\|\leq &c_0\sum_{j=1}^{k-1} \|P^{k-1-j} N_j\Delta_j^\delta\|
\leq \sqrt{n}c_0\sum_{j=1}^{k-1} \|P^{k-1-j} B^\frac12\|\|A\Delta_j^\delta\|\\
\leq &(c_1+c_2\delta)\sqrt{n}M\Bigg(c_0\|B\|^\frac12  (k-1+M)^{-1}+\sqrt{\frac{c_0}{2}}\sum_{j=1}^{k-2} \big(k-1-j\big)^{-\frac12} (j+M)^{-1}\Bigg).
\end{align*} 
Further, the assumption $\delta\leq 1$ and the estimate \eqref{eqn:sum_nu_1} in Lemma \ref{lem:sums} with $k-1$ and $\nu=0$ imply
\begin{align*}
\|q_{k}^\delta\|\leq &(c_1+c_2\delta)\sqrt{n}M\Bigg(c_0\|B\|^\frac12  (k-1)^{-1}+9\sqrt{\frac{c_0}{2}}(k-1)^{-\frac{1}{4}}\Bigg)\\
\leq &(c_1+c_2)\Big(c_0\|B\|^\frac12 +7 c_0^\frac12\Big)\sqrt{n}M(k-1)^{-\frac{1}{4}}.
\end{align*}
Similarly, by Lemmas \ref{lem:kernel}, \ref{lem:N} and \ref{lem:Delta} (when $c_0<\overline{C_0}$), and the estimate \eqref{eqn:sum_1_2} in Lemma \ref{lem:sums} with $k-1$, we derive
\begin{align*}
\E[\|q_{k}^\delta\|^2]\leq&  c_0^2\E[\|P^{k-1-j} N_j\Delta_j^\delta\|^2]
\leq c_0^2\sum_{j=1}^{k-1}\|P^{k-1-j} B^{\frac12}\|^2\E[\|A\Delta_j^\delta\|^2]\\
\leq& c_0(c_1+c_2) M^2 \bigg(c_0\|B\|(k-1+M)^{-2}+\frac{1}2 \sum_{j=1}^{k-2}  \big(k-1-j\big)^{-1}(j+M)^{-2}\bigg)\\
\leq& c_0(c_1+c_2) (c_0\|B\|+2) M^2 \big(k-1\big)^{-1}.
\end{align*}
This complete the proof of the lemma.
\end{proof}

\bibliographystyle{abbrv}
\bibliography{sgd}
\end{document}